\documentclass[12pt]{amsart}

\usepackage{tikz}

\usepackage[margin=1in,letterpaper,portrait]{geometry}

\theoremstyle{definition}
\newtheorem{thm}{Theorem}[section]
\newtheorem{prop}[thm]{Proposition}

\newtheorem{lemma}[thm]{Lemma}
\newtheorem{defn}[thm]{Definition}
\newtheorem{cor}[thm]{Corollary}
\newtheorem{obs}[thm]{Observation}

\newtheorem{rem}[thm]{Remark}

\DeclareMathOperator{\Des}{Des}
\DeclareMathOperator{\des}{des}

\DeclareMathOperator{\pk}{pk}
\DeclareMathOperator{\lpk}{lpk}

\DeclareMathOperator{\lp}{\leq^+}

\DeclareMathOperator{\len}{\leq^{--}}
\DeclareMathOperator{\gn}{\geq^{--}}
\DeclareMathOperator{\op}{\Omega}

\DeclareMathOperator{\A}{\mathcal{A}}
\DeclareMathOperator{\LL}{\mathcal{L}}
\DeclareMathOperator{\sep}{sep}
\DeclareMathOperator{\Prob}{Pr}

\newcommand{\ZZ}{\mathbb{Z}}
\newcommand{\NN}{\mathbb{N}}

\title{Card shuffling and $P$-partitions}

\author[Fulman]{Jason Fulman}%
\address{Department of Mathematics, USC, Los Angeles, CA}
\email{fulman@usc.edu}

\author[Petersen]{T. Kyle Petersen}%
\address{Department of Mathematical Sciences, DePaul University, Chicago, IL}
\email{tpeter21@depaul.edu}

\thanks{Fulman is supported by Simons Foundation Grant 400528. Petersen is supported by Simons Foundation Collaboration Travel Grant 353772. We thank Persi Diaconis for discussions about shuffling, and the referees for helpful comments.}

\date{April 15, 2021}

\begin{document}

\begin{abstract}
In this expository article, we highlight the direct connection between card shuffling and the functions known as $P$-partitions that come from algebraic combinatorics. While many (but not all) of the results we discuss are known, we give a unified treatment. The key idea is this: the probability of obtaining a permutation $\pi$ from shelf shuffling is the probability that a random $P$-partition is sorted by $\pi$, and the probability of obtaining $\pi$ from riffle shuffling is the probability that a random $P$-partition is sorted by $\pi^{-1}$.
\end{abstract}

\maketitle

\section{Introduction}

Methods for mixing a deck of playing cards have probably existed for as long as humans have played card games. Having a well-mixed deck is a central tenet of fair game play, while ill-mixed decks lead to subtle advantages for savvy players. The common English term for mixing a deck of cards is \emph{shuffling}.

In many parts of the world, the standard method for shuffling a deck of cards is to cut the deck into two (or more) piles and then to interleave the piles, with the cards in each pile staying in the same relative order. In the mathematical literature this type of shuffle is now modeled with the \emph{riffle shuffle}, first developed by Gilbert and Shannon in the 1950s for Bell Labs, and independently by Reeds in an unpublished manuscript from 1981. Bayer and Diaconis popularized the riffle shuffle with their landmark paper \cite{BD}. We will define the riffle shuffle precisely in Section \ref{sec:equivalence}.

Another type of shuffling, known as \emph{shelf shuffling}, is used in casinos. In this case a machine with a number of shelves mixes cards. This machine places cards one at a time onto a random shelf. Each shelf ends up with a small pile of cards and these piles are then removed and placed atop one another to form the mixed deck. In their 2013 paper \cite{DFH}, Diaconis, Fulman, and Holmes analyzed a mathematical model for shelf shuffling. We will precisely define shelf shuffling in Section \ref{sec:equivalence} as well.

Ever since the Bayer and Diaconis paper on riffle shuffling \cite{BD}, it has been well-understood that permutation statistics, such as the number of \emph{descents}, play an important role in understanding card shuffling. In particular, a key identity used in \cite{BD} to study the Gilbert-Shannon-Reeds riffle shuffle is the following: for any permutation $\pi$ in $S_n$,
\begin{equation}\label{eq:permsum}
 \binom{kl+n-\des(\pi)-1}{n} = \sum_{\sigma\tau=\pi} \binom{k + n-\des(\sigma) -1}{n} \binom{l+n-\des(\tau)-1}{n},
\end{equation}
where $\des(\pi) = |\{i: \pi(i)>\pi(i+1)\}|$ is the number of descents of $\pi$.

While the identity \eqref{eq:permsum} may be given various proofs
(the earliest of which is perhaps the one found in \cite{MP}), one of the nicest of these follows from work of Gessel \cite{Ge} in 1984, using the theory of \emph{$P$-partitions} \cite{Ge}. The ``$P$'' in $P$-partition stands for ``partially ordered set'' or ``poset.'' Stanley defined $P$-partitions as a way to generalize integer partitions to study plane partitions, but they can also be used to give a combinatorial framework for the study of symmetric and quasisymmetric functions, with applications to permutation enumeration. See Gessel's survey \cite{Ge:survey}, \cite[Section 4.5]{EC1}, and \cite[Section 7.19]{EC2}.

Identities similar to Equation \eqref{eq:permsum} show up in work of Petersen \cite{Pet} from 2007 using a slightly more general notion of $P$-partition that includes but also builds on Stembridge's notion of enriched $P$-partitions \cite{Stem}. In these cases, it is the number of \emph{peaks} of a permutation that matter (instances in which $\pi(i-1)<\pi(i)>\pi(i+1)$), rather than the number of descents.

The number of peaks proved crucial to the analysis of the shelf shuffler machine studied by Diaconis, Fulman, and Holmes \cite{DFH}. (In fact, they used a $P$-partition argument to prove one of their main results \cite[Theorem 3.2]{DFH}.) In this paper, we will use $P$-partitions to study a family of related shuffling schemes. Although the paper \cite{Dfor} mentions a connection between riffle shuffling and $P$-partitions, this does not seem to be widely known or explored (Gessel's survey \cite{Ge:survey} of $P$-partitions says nothing about the connection to shuffling).

We now describe a general framework used in the analysis of card shuffling. Each shuffling scheme gives a family of probability distributions on the set of permutations. Let $\Prob(m;\pi)$ denote any of these probability distributions on $S_n$ coming from shelf shuffling with $m$ shelves, or from riffle shuffling with $m$ piles. In the group algebra we define the generating function:
\[
 \phi_{\Prob}(m,n) = \sum_{\pi \in S_n} \Prob(m;\pi)\cdot \pi.
\]
Repeated shuffles correspond to the multiplying $\phi_{\Prob}(m)$ by itself:
\begin{align*}
\phi_{\Prob}(m,n)^2 &= \left(\sum_{\sigma \in S_n} \Prob(m;\sigma)\cdot\sigma\right)\left(\sum_{\tau \in S_n} \Prob(m;\tau)\cdot\tau\right),\\
&= \sum_{\pi \in S_n} \left( \sum_{\sigma \in S_n} \Prob(m;\sigma)\Prob(m;\sigma^{-1}\pi) \right) \pi.
\end{align*}
It transpires that each shuffling scheme we study generates an ergodic Markov chain, so with repeated shuffling we have convergence to a unique stationary distribution. Moreover, this distribution turns out to be the uniform distribution in each case:
\[
 \phi_{\Prob}(m,n)^k \to \sum_{\pi \in S_n} \frac{1}{n!} \pi \quad \mbox{ as $k\to \infty$.}\footnote{Such convergence to uniformity occurs under very mild conditions, as Aldous and Diaconis explain in \cite{AD}.}
\]

We can explain how Equation \eqref{eq:permsum} is relevant now. For the $m$-riffle shuffle studied by Bayer and Diaconis \cite{BD} ($m$ is the number of piles riffled together), we have
\[
 \Prob(m;\pi) = \frac{\binom{m + n-\des(\pi^{-1}) -1}{n}}{m^n},
\]
and thus \eqref{eq:permsum} implies that $\phi_{\Prob}(m,n)^2 = \phi_{\Prob}(m^2,n)$, and for larger $k$, $\phi_{\Prob}(m,n)^k = \phi_{\Prob}(m^k,n)$. This means that to analyze repeated shuffles, it suffices to study just one shuffle, but for an arbitrary number of piles.

This line of reasoning carries through for the probability distributions coming from other shuffling schemes connected to $P$-partitions, as we will explain later in the article. We will reproduce key results in both the classical riffle shuffle and shelf shuffler settings, e.g.,
\begin{itemize}
\item probability formulas,
\item convolution properties, and
\item convergence estimates.
\end{itemize}
Moreover, we will give analogous new results that use \emph{left enriched $P$-partitions} to analyze a ``lazy'' shelf shuffler and a corresponding riffle shuffle. In this new situation, the key permutation statistic is the number of \emph{left peaks}, for which we will obtain a recent enumerative result of Gessel and Zhuang \cite{GeZh} about the distribution of left peaks according to cycle type.

\begin{rem}[Other shuffles]
There are other mathematical models for shuffling cards that have been studied that we will not revisit in this paper, such as ``shuffles with a cut'' \cite{F0}, ``top-to-random shuffles'' \cite{BHR, BrD}, and the ``overhand shuffle'' \cite{Pem}. We also mention that the term ``shuffle'' is often used in algebraic combinatorics to mean the multiset of all interleavings of two words. See, e.g., \cite{LR}. While this notion of a shuffle is analogous to the riffle shuffle (and can in fact be useful in the study of card shuffling) it is not what we mean by a shuffle in this paper.
\end{rem}

\subsection*{Organization of the paper}

In Section \ref{sec:equivalence}, we will show how shuffling and random sampling of $P$-partitions are equivalent. In Section \ref{sec:enum}, we survey enumerative results in the $P$-partition literature and translate them into probabilistic statements about shuffling. In Section \ref{sec:converge}, we give convergence estimates for shuffling, and in Section \ref{sec:cycle}, we study the distribution of
cycle structure for lazy shelf shufflers.

\section{Equivalence of Shelf Shuffling and $P$-partitions}\label{sec:equivalence}

In this section we establish the direct link between $P$-partitions and shuffling.

\subsection{Shelf shuffling}

We will now describe a new method of shuffling that we call \emph{lazy shelf shuffling}, along with the method of shelf shuffling studied by Diaconis, Fulman, and Holmes \cite{DFH} and inverse riffle shuffling studied by Bayer and Diaconis \cite{BD}. Let $n$ denote the number of cards in the deck, and suppose the cards are labeled $1,2,\ldots,n$ from top to bottom. We quote here from the description in \cite{DFH} of an actual machine with 10 shelves, manufactured for use in casinos:
\begin{quote}
A deck of cards is dropped into the top of the box. An internal elevator moves the deck up and down within the box. Cards are sequentially dealt from the bottom of the deck onto the shelves; shelves are chosen uniformly at random at the command of a random number generator. Each card is randomly placed above or below previous cards on the shelf with probability 1/2. At the end, each shelf contains about 1/10 of the deck. The ten piles are now assembled into one pile, in random order.
\end{quote}

We modify this description only slightly. First, we allow any fixed number $m$ to be the number of shelves onto which we will place the cards. (The actual machine has $m=10$.) This was of course done in \cite{DFH}. Second, for convenience, we also assume the cards are sequentially dealt \emph{from the top of the deck}, rather than from the bottom, i.e., we place card 1 first, then card 2, and so on. This choice makes some difference in the combinatorial details (allowing us to work with peaks rather than valleys), but little difference in the statistical analysis. See Remark \ref{rem:topvbot}. Third, note the final step of assembling the piles into random order is superfluous, so we put the cards on the first shelf on top, followed by the cards on the second shelf, etc.

A more significant difference from the standard shelf shuffler is in the addition of another shelf, at the top of the box, onto which cards may only be placed below previously placed cards.

To get our three card shuffling schemes from our imaginary machine, we install a control panel with buttons that can be used to direct the machine to shuffle in one of three modes. We pretend there are three buttons on the front of the machine, labeled \textbf{LAZY}, \textbf{STANDARD}, and \textbf{STRICT}. Here is a description of each operating mode.

\begin{itemize}
\item \textbf{Strict mode}. When the machine is in strict mode, it only places cards below cards that are already on a shelf. In this way, strict mode only has to choose a random shelf for each card, each with probability $1/m$. Thus each way of assigning the cards to the shelves occurs with probability $1/m^n$. For example, with $m=3$ piles, and $n=9$ cards, we might shuffle cards as shown in Table \ref{tab:riffle}. We obtain the permutation $\pi = 234569178$ by reading the card labels from the top of the top shelf to the bottom of the bottom shelf. We remark that different ways of assigning the cards to the shelves can produce the same permutation, e.g., we could have inserted card $4$ in shelf $1$ and obtained the same permutation.

\begin{table}
\begin{tikzpicture}
\draw(-.2,-.75) node[left] {Card:};
\draw(-.2,-1.25) node[left] {Shelf:};
\draw(-.2,6.26) node[left] {Shelf 1:};
\draw(-.2,3.75) node[left] {Shelf 2:};
\draw(-.2,1.25) node[left] {Shelf 3:};
\draw (.5,-1.25) node {$3$};
\draw (.5,-.75) node {$1$};
\draw (2,-1.25) node {$1$};
\draw (2,-.75) node {$2$};
\draw (3.5,-1.25) node {$1$};
\draw (3.5,-.75) node {$3$};
\draw (5,-1.25) node {$2$};
\draw (5,-.75) node {$4$};
\draw (6.5,-1.25) node {$2$};
\draw (6.5,-.75) node {$5$};
\draw (8,-1.25) node {$2$};
\draw (8,-.75) node {$6$};
\draw (9.5,-1.25) node {$3$};
\draw (9.5,-.75) node {$7$};
\draw (11,-1.25) node {$3$};
\draw (11,-.75) node {$8$};
\draw (12.5,-1.25) node {$2$};
\draw (12.5,-.75) node {$9$};
\draw [line width=4] (-.5,-.25)--(13.5,-.25);
\draw(.5,0) node[above] {
 \begin{tikzpicture}[baseline=0]
  \draw[line width=3, rounded corners] (0,0)--(0,7.5)--(1.5,7.5)--(1.5,0)--(0,0)--cycle;
  \draw[line width=3] (0,2.5)--(1.5,2.5);
  \draw[line width=3] (0,5)--(1.5,5);
  \draw(.75,0) node{
    \begin{tikzpicture}[baseline=0]
     \draw[fill=white!80!black] (0.3,0.15)--(0.6,1.1)--(1.2,1.1)--(0.9,0.15)--(0.3,0.15)--cycle;
     \draw (.98,0.76) node[scale=.75, rotate=-12] {$1$};
     \draw (0.45,0.12) node[scale=.75, rotate=-12] {$1$};
     \end{tikzpicture}
     };
 \end{tikzpicture}
};
\draw(2,0) node[above] {
 \begin{tikzpicture}[baseline=0]
  \draw[line width=3, rounded corners] (0,0)--(0,7.5)--(1.5,7.5)--(1.5,0)--(0,0)--cycle;
  \draw[line width=3] (0,2.5)--(1.5,2.5);
  \draw[line width=3] (0,5)--(1.5,5);
  \draw(.75,5) node{
    \begin{tikzpicture}[baseline=0]
     \draw[fill=white!80!black] (0.3,0.15)--(0.6,1.1)--(1.2,1.1)--(0.9,0.15)--(0.3,0.15)--cycle;
     \draw (.98,0.76) node[scale=.75, rotate=-12] {$2$};
     \draw (0.45,0.12) node[scale=.75, rotate=-12] {$2$};
     \end{tikzpicture}
     };
  \draw(.75,0) node{
    \begin{tikzpicture}[baseline=0]
     \draw[fill=white] (0.3,0.15)--(0.6,1.1)--(1.2,1.1)--(0.9,0.15)--(0.3,0.15)--cycle;
     \draw (.98,0.76) node[scale=.75, rotate=-12] {$1$};
     \draw (0.45,0.12) node[scale=.75, rotate=-12] {$1$};
     \end{tikzpicture}
     };
 \end{tikzpicture}
};
\draw(3.5,0) node[above] {
 \begin{tikzpicture}[baseline=0]
  \draw[line width=3, rounded corners] (0,0)--(0,7.5)--(1.5,7.5)--(1.5,0)--(0,0)--cycle;
  \draw[line width=3] (0,2.5)--(1.5,2.5);
  \draw[line width=3] (0,5)--(1.5,5);
  \draw(.75,5) node{
    \begin{tikzpicture}[baseline=0]
     \draw[fill=white!80!black] (0.3,0.15)--(0.6,1.1)--(1.2,1.1)--(0.9,0.15)--(0.3,0.15)--cycle;
     \draw (.98,0.76) node[scale=.75, rotate=-12] {$3$};
     \draw (0.45,0.12) node[scale=.75, rotate=-12] {$3$};
     \end{tikzpicture}
     };
  \draw(.75,5+.35) node{
    \begin{tikzpicture}[baseline=0]
     \draw[fill=white] (0.3,0.15)--(0.6,1.1)--(1.2,1.1)--(0.9,0.15)--(0.3,0.15)--cycle;
     \draw (.98,0.76) node[scale=.75, rotate=-12] {$2$};
     \draw (0.45,0.12) node[scale=.75, rotate=-12] {$2$};
     \end{tikzpicture}
     };
  \draw(.75,0) node{
    \begin{tikzpicture}[baseline=0]
     \draw[fill=white] (0.3,0.15)--(0.6,1.1)--(1.2,1.1)--(0.9,0.15)--(0.3,0.15)--cycle;
     \draw (.98,0.76) node[scale=.75, rotate=-12] {$1$};
     \draw (0.45,0.12) node[scale=.75, rotate=-12] {$1$};
     \end{tikzpicture}
     };
 \end{tikzpicture}
};
\draw(5,0) node[above] {
 \begin{tikzpicture}[baseline=0]
  \draw[line width=3, rounded corners] (0,0)--(0,7.5)--(1.5,7.5)--(1.5,0)--(0,0)--cycle;
  \draw[line width=3] (0,2.5)--(1.5,2.5);
  \draw[line width=3] (0,5)--(1.5,5);
  \draw(.75,5) node{
    \begin{tikzpicture}[baseline=0]
     \draw[fill=white] (0.3,0.15)--(0.6,1.1)--(1.2,1.1)--(0.9,0.15)--(0.3,0.15)--cycle;
     \draw (.98,0.76) node[scale=.75, rotate=-12] {$3$};
     \draw (0.45,0.12) node[scale=.75, rotate=-12] {$3$};
     \end{tikzpicture}
     };
  \draw(.75,5+.35) node{
    \begin{tikzpicture}[baseline=0]
     \draw[fill=white] (0.3,0.15)--(0.6,1.1)--(1.2,1.1)--(0.9,0.15)--(0.3,0.15)--cycle;
     \draw (.98,0.76) node[scale=.75, rotate=-12] {$2$};
     \draw (0.45,0.12) node[scale=.75, rotate=-12] {$2$};
     \end{tikzpicture}
     };
  \draw(.75,2.5) node{
    \begin{tikzpicture}[baseline=0]
     \draw[fill=white!80!black] (0.3,0.15)--(0.6,1.1)--(1.2,1.1)--(0.9,0.15)--(0.3,0.15)--cycle;
     \draw (.98,0.76) node[scale=.75, rotate=-12] {$4$};
     \draw (0.45,0.12) node[scale=.75, rotate=-12] {$4$};
     \end{tikzpicture}
     };
  \draw(.75,0) node{
    \begin{tikzpicture}[baseline=0]
     \draw[fill=white] (0.3,0.15)--(0.6,1.1)--(1.2,1.1)--(0.9,0.15)--(0.3,0.15)--cycle;
     \draw (.98,0.76) node[scale=.75, rotate=-12] {$1$};
     \draw (0.45,0.12) node[scale=.75, rotate=-12] {$1$};
     \end{tikzpicture}
     };
 \end{tikzpicture}
};
\draw(6.5,0) node[above] {
 \begin{tikzpicture}[baseline=0]
  \draw[line width=3, rounded corners] (0,0)--(0,7.5)--(1.5,7.5)--(1.5,0)--(0,0)--cycle;
  \draw[line width=3] (0,2.5)--(1.5,2.5);
  \draw[line width=3] (0,5)--(1.5,5);
  \draw(.75,5) node{
    \begin{tikzpicture}[baseline=0]
     \draw[fill=white] (0.3,0.15)--(0.6,1.1)--(1.2,1.1)--(0.9,0.15)--(0.3,0.15)--cycle;
     \draw (.98,0.76) node[scale=.75, rotate=-12] {$3$};
     \draw (0.45,0.12) node[scale=.75, rotate=-12] {$3$};
     \end{tikzpicture}
     };
  \draw(.75,5+.35) node{
    \begin{tikzpicture}[baseline=0]
     \draw[fill=white] (0.3,0.15)--(0.6,1.1)--(1.2,1.1)--(0.9,0.15)--(0.3,0.15)--cycle;
     \draw (.98,0.76) node[scale=.75, rotate=-12] {$2$};
     \draw (0.45,0.12) node[scale=.75, rotate=-12] {$2$};
     \end{tikzpicture}
     };
  \draw(.75,2.5) node{
    \begin{tikzpicture}[baseline=0]
     \draw[fill=white!80!black] (0.3,0.15)--(0.6,1.1)--(1.2,1.1)--(0.9,0.15)--(0.3,0.15)--cycle;
     \draw (.98,0.76) node[scale=.75, rotate=-12] {$5$};
     \draw (0.45,0.12) node[scale=.75, rotate=-12] {$5$};
     \end{tikzpicture}
     };
  \draw(.75,2.5+.35) node{
    \begin{tikzpicture}[baseline=0]
     \draw[fill=white] (0.3,0.15)--(0.6,1.1)--(1.2,1.1)--(0.9,0.15)--(0.3,0.15)--cycle;
     \draw (.98,0.76) node[scale=.75, rotate=-12] {$4$};
     \draw (0.45,0.12) node[scale=.75, rotate=-12] {$4$};
     \end{tikzpicture}
     };
  \draw(.75,0) node{
    \begin{tikzpicture}[baseline=0]
     \draw[fill=white] (0.3,0.15)--(0.6,1.1)--(1.2,1.1)--(0.9,0.15)--(0.3,0.15)--cycle;
     \draw (.98,0.76) node[scale=.75, rotate=-12] {$1$};
     \draw (0.45,0.12) node[scale=.75, rotate=-12] {$1$};
     \end{tikzpicture}
     };
 \end{tikzpicture}
};
\draw(8,0) node[above] {
 \begin{tikzpicture}[baseline=0]
  \draw[line width=3, rounded corners] (0,0)--(0,7.5)--(1.5,7.5)--(1.5,0)--(0,0)--cycle;
  \draw[line width=3] (0,2.5)--(1.5,2.5);
  \draw[line width=3] (0,5)--(1.5,5);
  \draw(.75,5) node{
    \begin{tikzpicture}[baseline=0]
     \draw[fill=white] (0.3,0.15)--(0.6,1.1)--(1.2,1.1)--(0.9,0.15)--(0.3,0.15)--cycle;
     \draw (.98,0.76) node[scale=.75, rotate=-12] {$3$};
     \draw (0.45,0.12) node[scale=.75, rotate=-12] {$3$};
     \end{tikzpicture}
     };
  \draw(.75,5+.35) node{
    \begin{tikzpicture}[baseline=0]
     \draw[fill=white] (0.3,0.15)--(0.6,1.1)--(1.2,1.1)--(0.9,0.15)--(0.3,0.15)--cycle;
     \draw (.98,0.76) node[scale=.75, rotate=-12] {$2$};
     \draw (0.45,0.12) node[scale=.75, rotate=-12] {$2$};
     \end{tikzpicture}
     };
  \draw(.75,2.5) node{
    \begin{tikzpicture}[baseline=0]
     \draw[fill=white!80!black] (0.3,0.15)--(0.6,1.1)--(1.2,1.1)--(0.9,0.15)--(0.3,0.15)--cycle;
     \draw (.98,0.76) node[scale=.75, rotate=-12] {$6$};
     \draw (0.45,0.12) node[scale=.75, rotate=-12] {$6$};
     \end{tikzpicture}
     };
  \draw(.75,2.5+.35) node{
    \begin{tikzpicture}[baseline=0]
     \draw[fill=white] (0.3,0.15)--(0.6,1.1)--(1.2,1.1)--(0.9,0.15)--(0.3,0.15)--cycle;
     \draw (.98,0.76) node[scale=.75, rotate=-12] {$5$};
     \draw (0.45,0.12) node[scale=.75, rotate=-12] {$5$};
     \end{tikzpicture}
     };
  \draw(.75,2.5+.7) node{
    \begin{tikzpicture}[baseline=0]
     \draw[fill=white] (0.3,0.15)--(0.6,1.1)--(1.2,1.1)--(0.9,0.15)--(0.3,0.15)--cycle;
     \draw (.98,0.76) node[scale=.75, rotate=-12] {$4$};
     \draw (0.45,0.12) node[scale=.75, rotate=-12] {$4$};
     \end{tikzpicture}
     };
  \draw(.75,0) node{
    \begin{tikzpicture}[baseline=0]
     \draw[fill=white] (0.3,0.15)--(0.6,1.1)--(1.2,1.1)--(0.9,0.15)--(0.3,0.15)--cycle;
     \draw (.98,0.76) node[scale=.75, rotate=-12] {$1$};
     \draw (0.45,0.12) node[scale=.75, rotate=-12] {$1$};
     \end{tikzpicture}
     };
 \end{tikzpicture}
};
\draw(9.5,0) node[above] {
 \begin{tikzpicture}[baseline=0]
  \draw[line width=3, rounded corners] (0,0)--(0,7.5)--(1.5,7.5)--(1.5,0)--(0,0)--cycle;
  \draw[line width=3] (0,2.5)--(1.5,2.5);
  \draw[line width=3] (0,5)--(1.5,5);
  \draw(.75,5) node{
    \begin{tikzpicture}[baseline=0]
     \draw[fill=white] (0.3,0.15)--(0.6,1.1)--(1.2,1.1)--(0.9,0.15)--(0.3,0.15)--cycle;
     \draw (.98,0.76) node[scale=.75, rotate=-12] {$3$};
     \draw (0.45,0.12) node[scale=.75, rotate=-12] {$3$};
     \end{tikzpicture}
     };
  \draw(.75,5+.35) node{
    \begin{tikzpicture}[baseline=0]
     \draw[fill=white] (0.3,0.15)--(0.6,1.1)--(1.2,1.1)--(0.9,0.15)--(0.3,0.15)--cycle;
     \draw (.98,0.76) node[scale=.75, rotate=-12] {$2$};
     \draw (0.45,0.12) node[scale=.75, rotate=-12] {$2$};
     \end{tikzpicture}
     };
  \draw(.75,2.5) node{
    \begin{tikzpicture}[baseline=0]
     \draw[fill=white] (0.3,0.15)--(0.6,1.1)--(1.2,1.1)--(0.9,0.15)--(0.3,0.15)--cycle;
     \draw (.98,0.76) node[scale=.75, rotate=-12] {$6$};
     \draw (0.45,0.12) node[scale=.75, rotate=-12] {$6$};
     \end{tikzpicture}
     };
  \draw(.75,2.5+.35) node{
    \begin{tikzpicture}[baseline=0]
     \draw[fill=white] (0.3,0.15)--(0.6,1.1)--(1.2,1.1)--(0.9,0.15)--(0.3,0.15)--cycle;
     \draw (.98,0.76) node[scale=.75, rotate=-12] {$5$};
     \draw (0.45,0.12) node[scale=.75, rotate=-12] {$5$};
     \end{tikzpicture}
     };
  \draw(.75,2.5+.7) node{
    \begin{tikzpicture}[baseline=0]
     \draw[fill=white] (0.3,0.15)--(0.6,1.1)--(1.2,1.1)--(0.9,0.15)--(0.3,0.15)--cycle;
     \draw (.98,0.76) node[scale=.75, rotate=-12] {$4$};
     \draw (0.45,0.12) node[scale=.75, rotate=-12] {$4$};
     \end{tikzpicture}
     };
  \draw(.75,0) node{
    \begin{tikzpicture}[baseline=0]
     \draw[fill=white!80!black] (0.3,0.15)--(0.6,1.1)--(1.2,1.1)--(0.9,0.15)--(0.3,0.15)--cycle;
     \draw (.98,0.76) node[scale=.75, rotate=-12] {$7$};
     \draw (0.45,0.12) node[scale=.75, rotate=-12] {$7$};
     \end{tikzpicture}
     };
  \draw(.75,.35) node{
    \begin{tikzpicture}[baseline=0]
     \draw[fill=white] (0.3,0.15)--(0.6,1.1)--(1.2,1.1)--(0.9,0.15)--(0.3,0.15)--cycle;
     \draw (.98,0.76) node[scale=.75, rotate=-12] {$1$};
     \draw (0.45,0.12) node[scale=.75, rotate=-12] {$1$};
     \end{tikzpicture}
     };
 \end{tikzpicture}
};
\draw(11,0) node[above] {
 \begin{tikzpicture}[baseline=0]
  \draw[line width=3, rounded corners] (0,0)--(0,7.5)--(1.5,7.5)--(1.5,0)--(0,0)--cycle;
  \draw[line width=3] (0,2.5)--(1.5,2.5);
  \draw[line width=3] (0,5)--(1.5,5);
  \draw(.75,5) node{
    \begin{tikzpicture}[baseline=0]
     \draw[fill=white] (0.3,0.15)--(0.6,1.1)--(1.2,1.1)--(0.9,0.15)--(0.3,0.15)--cycle;
     \draw (.98,0.76) node[scale=.75, rotate=-12] {$3$};
     \draw (0.45,0.12) node[scale=.75, rotate=-12] {$3$};
     \end{tikzpicture}
     };
  \draw(.75,5+.35) node{
    \begin{tikzpicture}[baseline=0]
     \draw[fill=white] (0.3,0.15)--(0.6,1.1)--(1.2,1.1)--(0.9,0.15)--(0.3,0.15)--cycle;
     \draw (.98,0.76) node[scale=.75, rotate=-12] {$2$};
     \draw (0.45,0.12) node[scale=.75, rotate=-12] {$2$};
     \end{tikzpicture}
     };
  \draw(.75,2.5) node{
    \begin{tikzpicture}[baseline=0]
     \draw[fill=white] (0.3,0.15)--(0.6,1.1)--(1.2,1.1)--(0.9,0.15)--(0.3,0.15)--cycle;
     \draw (.98,0.76) node[scale=.75, rotate=-12] {$6$};
     \draw (0.45,0.12) node[scale=.75, rotate=-12] {$6$};
     \end{tikzpicture}
     };
  \draw(.75,2.5+.35) node{
    \begin{tikzpicture}[baseline=0]
     \draw[fill=white] (0.3,0.15)--(0.6,1.1)--(1.2,1.1)--(0.9,0.15)--(0.3,0.15)--cycle;
     \draw (.98,0.76) node[scale=.75, rotate=-12] {$5$};
     \draw (0.45,0.12) node[scale=.75, rotate=-12] {$5$};
     \end{tikzpicture}
     };
  \draw(.75,2.5+.7) node{
    \begin{tikzpicture}[baseline=0]
     \draw[fill=white] (0.3,0.15)--(0.6,1.1)--(1.2,1.1)--(0.9,0.15)--(0.3,0.15)--cycle;
     \draw (.98,0.76) node[scale=.75, rotate=-12] {$4$};
     \draw (0.45,0.12) node[scale=.75, rotate=-12] {$4$};
     \end{tikzpicture}
     };
  \draw(.75,0) node{
    \begin{tikzpicture}[baseline=0]
     \draw[fill=white!80!black] (0.3,0.15)--(0.6,1.1)--(1.2,1.1)--(0.9,0.15)--(0.3,0.15)--cycle;
     \draw (.98,0.76) node[scale=.75, rotate=-12] {$8$};
     \draw (0.45,0.12) node[scale=.75, rotate=-12] {$8$};
     \end{tikzpicture}
     };
  \draw(.75,.35) node{
    \begin{tikzpicture}[baseline=0]
     \draw[fill=white] (0.3,0.15)--(0.6,1.1)--(1.2,1.1)--(0.9,0.15)--(0.3,0.15)--cycle;
     \draw (.98,0.76) node[scale=.75, rotate=-12] {$7$};
     \draw (0.45,0.12) node[scale=.75, rotate=-12] {$7$};
     \end{tikzpicture}
     };
  \draw(.75,.7) node{
    \begin{tikzpicture}[baseline=0]
     \draw[fill=white] (0.3,0.15)--(0.6,1.1)--(1.2,1.1)--(0.9,0.15)--(0.3,0.15)--cycle;
     \draw (.98,0.76) node[scale=.75, rotate=-12] {$1$};
     \draw (0.45,0.12) node[scale=.75, rotate=-12] {$1$};
     \end{tikzpicture}
     };
 \end{tikzpicture}
};
\draw(12.5,0) node[above] {
 \begin{tikzpicture}[baseline=0]
  \draw[line width=3, rounded corners] (0,0)--(0,7.5)--(1.5,7.5)--(1.5,0)--(0,0)--cycle;
  \draw[line width=3] (0,2.5)--(1.5,2.5);
  \draw[line width=3] (0,5)--(1.5,5);
  \draw(.75,5) node{
    \begin{tikzpicture}[baseline=0]
     \draw[fill=white] (0.3,0.15)--(0.6,1.1)--(1.2,1.1)--(0.9,0.15)--(0.3,0.15)--cycle;
     \draw (.98,0.76) node[scale=.75, rotate=-12] {$3$};
     \draw (0.45,0.12) node[scale=.75, rotate=-12] {$3$};
     \end{tikzpicture}
     };
  \draw(.75,5+.35) node{
    \begin{tikzpicture}[baseline=0]
     \draw[fill=white] (0.3,0.15)--(0.6,1.1)--(1.2,1.1)--(0.9,0.15)--(0.3,0.15)--cycle;
     \draw (.98,0.76) node[scale=.75, rotate=-12] {$2$};
     \draw (0.45,0.12) node[scale=.75, rotate=-12] {$2$};
     \end{tikzpicture}
     };
  \draw(.75,2.5) node{
    \begin{tikzpicture}[baseline=0]
     \draw[fill=white!80!black] (0.3,0.15)--(0.6,1.1)--(1.2,1.1)--(0.9,0.15)--(0.3,0.15)--cycle;
     \draw (.98,0.76) node[scale=.75, rotate=-12] {$9$};
     \draw (0.45,0.12) node[scale=.75, rotate=-12] {$9$};
     \end{tikzpicture}
     };
  \draw(.75,2.5+.35) node{
    \begin{tikzpicture}[baseline=0]
     \draw[fill=white] (0.3,0.15)--(0.6,1.1)--(1.2,1.1)--(0.9,0.15)--(0.3,0.15)--cycle;
     \draw (.98,0.76) node[scale=.75, rotate=-12] {$6$};
     \draw (0.45,0.12) node[scale=.75, rotate=-12] {$6$};
     \end{tikzpicture}
     };
  \draw(.75,2.5+.7) node{
    \begin{tikzpicture}[baseline=0]
     \draw[fill=white] (0.3,0.15)--(0.6,1.1)--(1.2,1.1)--(0.9,0.15)--(0.3,0.15)--cycle;
     \draw (.98,0.76) node[scale=.75, rotate=-12] {$5$};
     \draw (0.45,0.12) node[scale=.75, rotate=-12] {$5$};
     \end{tikzpicture}
     };
  \draw(.75,2.5+1.05) node{
    \begin{tikzpicture}[baseline=0]
     \draw[fill=white] (0.3,0.15)--(0.6,1.1)--(1.2,1.1)--(0.9,0.15)--(0.3,0.15)--cycle;
     \draw (.98,0.76) node[scale=.75, rotate=-12] {$4$};
     \draw (0.45,0.12) node[scale=.75, rotate=-12] {$4$};
     \end{tikzpicture}
     };
  \draw(.75,0) node{
    \begin{tikzpicture}[baseline=0]
     \draw[fill=white] (0.3,0.15)--(0.6,1.1)--(1.2,1.1)--(0.9,0.15)--(0.3,0.15)--cycle;
     \draw (.98,0.76) node[scale=.75, rotate=-12] {$8$};
     \draw (0.45,0.12) node[scale=.75, rotate=-12] {$8$};
     \end{tikzpicture}
     };
  \draw(.75,.35) node{
    \begin{tikzpicture}[baseline=0]
     \draw[fill=white] (0.3,0.15)--(0.6,1.1)--(1.2,1.1)--(0.9,0.15)--(0.3,0.15)--cycle;
     \draw (.98,0.76) node[scale=.75, rotate=-12] {$7$};
     \draw (0.45,0.12) node[scale=.75, rotate=-12] {$7$};
     \end{tikzpicture}
     };
  \draw(.75,.7) node{
    \begin{tikzpicture}[baseline=0]
     \draw[fill=white] (0.3,0.15)--(0.6,1.1)--(1.2,1.1)--(0.9,0.15)--(0.3,0.15)--cycle;
     \draw (.98,0.76) node[scale=.75, rotate=-12] {$1$};
     \draw (0.45,0.12) node[scale=.75, rotate=-12] {$1$};
     \end{tikzpicture}
     };
 \end{tikzpicture}
};
\end{tikzpicture}
\bigskip
\caption{A strict shuffle of $n=9$ cards into $m=3$ shelves resulting in the permutation $\pi = 234569178$. Recently placed cards are shaded.}\label{tab:riffle}
\end{table}

\item \textbf{Standard mode.} This is the card shuffling model studied in \cite{DFH}, except that cards are dealt from the top rather than the bottom. In this mode, for each card $i$ the machine chooses a shelf with probability $1/m$, then places card $i$ at the top of the stack of cards with probability $1/2$, and at the bottom of that stack with probability $1/2$. Alternatively, we can imagine choosing a single (shelf, side) pair with probability $1/2m$, where ``side'' can be ``top'' or ``bottom.'' Each outcome thus occurs with probability $1/(2m)^n$. For example, with $m=2$ piles, and $n=9$ cards, the machine might shuffle cards as shown in Table \ref{tab:shelf}. We identify this shuffle with the permutation $\pi = 981257436$. Again, the same permutation can result from a variety of different outcomes.

\begin{table}

\begin{tikzpicture}
\draw(-.2,-.75) node[left] {Card:};
\draw(-.2,-1.25) node[left] {Shelf:};
\draw(-.2,4.5) node[left] {Shelf 1:};
\draw(-.2,1.5) node[left] {Shelf 2:};
\draw (.5,-1.25) node {1t};
\draw (.5,-.75) node {$1$};
\draw (2,-1.25) node {1b};
\draw (2,-.75) node {$2$};
\draw (3.5,-1.25) node {2t};
\draw (3.5,-.75) node {$3$};
\draw (5,-1.25) node {2t};
\draw (5,-.75) node {$4$};
\draw (6.5,-1.25) node {1b};
\draw (6.5,-.75) node {$5$};
\draw (8,-1.25) node {2b};
\draw (8,-.75) node {$6$};
\draw (9.5,-1.25) node {2t};
\draw (9.5,-.75) node {$7$};
\draw (11,-1.25) node {1t};
\draw (11,-.75) node {$8$};
\draw (12.5,-1.25) node {1t};
\draw (12.5,-.75) node {$9$};
\draw [line width=4] (-.5,-.25)--(13.5,-.25);
\draw(.5,0) node[above] {
 \begin{tikzpicture}[baseline=0]
\draw[line width=3, rounded corners] (0,0)--(0,6)--(1.5,6)--(1.5,0)--(0,0)--cycle;
  \draw[line width=3] (0,3)--(1.5,3);
  \draw(.75,3) node{
    \begin{tikzpicture}[baseline=0]
     \draw[fill=white!80!black] (0.3,0.15)--(0.6,1.1)--(1.2,1.1)--(0.9,0.15)--(0.3,0.15)--cycle;
     \draw (.98,0.76) node[scale=.75, rotate=-12] {$1$};
     \draw (0.45,0.12) node[scale=.75, rotate=-12] {$1$};
     \end{tikzpicture}
     };
 \end{tikzpicture}
};
\draw(2,0) node[above] {
 \begin{tikzpicture}[baseline=0]
\draw[line width=3, rounded corners] (0,0)--(0,6)--(1.5,6)--(1.5,0)--(0,0)--cycle;
  \draw[line width=3] (0,3)--(1.5,3);
  \draw(.75,3) node{
    \begin{tikzpicture}[baseline=0]
     \draw[fill=white!80!black] (0.3,0.15)--(0.6,1.1)--(1.2,1.1)--(0.9,0.15)--(0.3,0.15)--cycle;
     \draw (.98,0.76) node[scale=.75, rotate=-12] {$2$};
     \draw (0.45,0.12) node[scale=.75, rotate=-12] {$2$};
     \end{tikzpicture}
     };
  \draw(.75,3+.35) node{
    \begin{tikzpicture}[baseline=0]
     \draw[fill=white] (0.3,0.15)--(0.6,1.1)--(1.2,1.1)--(0.9,0.15)--(0.3,0.15)--cycle;
     \draw (.98,0.76) node[scale=.75, rotate=-12] {$1$};
     \draw (0.45,0.12) node[scale=.75, rotate=-12] {$1$};
     \end{tikzpicture}
     };
 \end{tikzpicture}
};
\draw(3.5,0) node[above] {
 \begin{tikzpicture}[baseline=0]
\draw[line width=3, rounded corners] (0,0)--(0,6)--(1.5,6)--(1.5,0)--(0,0)--cycle;
  \draw[line width=3] (0,3)--(1.5,3);
  \draw(.75,3) node{
    \begin{tikzpicture}[baseline=0]
     \draw[fill=white] (0.3,0.15)--(0.6,1.1)--(1.2,1.1)--(0.9,0.15)--(0.3,0.15)--cycle;
     \draw (.98,0.76) node[scale=.75, rotate=-12] {$2$};
     \draw (0.45,0.12) node[scale=.75, rotate=-12] {$2$};
     \end{tikzpicture}
     };
  \draw(.75,3+.35) node{
    \begin{tikzpicture}[baseline=0]
     \draw[fill=white] (0.3,0.15)--(0.6,1.1)--(1.2,1.1)--(0.9,0.15)--(0.3,0.15)--cycle;
     \draw (.98,0.76) node[scale=.75, rotate=-12] {$1$};
     \draw (0.45,0.12) node[scale=.75, rotate=-12] {$1$};
     \end{tikzpicture}
     };
  \draw(.75,0) node{
    \begin{tikzpicture}[baseline=0]
     \draw[fill=white!80!black] (0.3,0.15)--(0.6,1.1)--(1.2,1.1)--(0.9,0.15)--(0.3,0.15)--cycle;
     \draw (.98,0.76) node[scale=.75, rotate=-12] {$3$};
     \draw (0.45,0.12) node[scale=.75, rotate=-12] {$3$};
     \end{tikzpicture}
     };
 \end{tikzpicture}
};
\draw(5,0) node[above] {
 \begin{tikzpicture}[baseline=0]
\draw[line width=3, rounded corners] (0,0)--(0,6)--(1.5,6)--(1.5,0)--(0,0)--cycle;
  \draw[line width=3] (0,3)--(1.5,3);
  \draw(.75,3) node{
    \begin{tikzpicture}[baseline=0]
     \draw[fill=white] (0.3,0.15)--(0.6,1.1)--(1.2,1.1)--(0.9,0.15)--(0.3,0.15)--cycle;
     \draw (.98,0.76) node[scale=.75, rotate=-12] {$2$};
     \draw (0.45,0.12) node[scale=.75, rotate=-12] {$2$};
     \end{tikzpicture}
     };
  \draw(.75,3+.35) node{
    \begin{tikzpicture}[baseline=0]
     \draw[fill=white] (0.3,0.15)--(0.6,1.1)--(1.2,1.1)--(0.9,0.15)--(0.3,0.15)--cycle;
     \draw (.98,0.76) node[scale=.75, rotate=-12] {$1$};
     \draw (0.45,0.12) node[scale=.75, rotate=-12] {$1$};
     \end{tikzpicture}
     };
  \draw(.75,0) node{
    \begin{tikzpicture}[baseline=0]
     \draw[fill=white] (0.3,0.15)--(0.6,1.1)--(1.2,1.1)--(0.9,0.15)--(0.3,0.15)--cycle;
     \draw (.98,0.76) node[scale=.75, rotate=-12] {$3$};
     \draw (0.45,0.12) node[scale=.75, rotate=-12] {$3$};
     \end{tikzpicture}
     };
  \draw(.75,.35) node{
    \begin{tikzpicture}[baseline=0]
     \draw[fill=white!80!black] (0.3,0.15)--(0.6,1.1)--(1.2,1.1)--(0.9,0.15)--(0.3,0.15)--cycle;
     \draw (.98,0.76) node[scale=.75, rotate=-12] {$4$};
     \draw (0.45,0.12) node[scale=.75, rotate=-12] {$4$};
     \end{tikzpicture}
     };
 \end{tikzpicture}
};
\draw(6.5,0) node[above] {
 \begin{tikzpicture}[baseline=0]
\draw[line width=3, rounded corners] (0,0)--(0,6)--(1.5,6)--(1.5,0)--(0,0)--cycle;
  \draw[line width=3] (0,3)--(1.5,3);
  \draw(.75,3) node{
    \begin{tikzpicture}[baseline=0]
     \draw[fill=white!80!black] (0.3,0.15)--(0.6,1.1)--(1.2,1.1)--(0.9,0.15)--(0.3,0.15)--cycle;
     \draw (.98,0.76) node[scale=.75, rotate=-12] {$5$};
     \draw (0.45,0.12) node[scale=.75, rotate=-12] {$5$};
     \end{tikzpicture}
     };
  \draw(.75,3+.35) node{
    \begin{tikzpicture}[baseline=0]
     \draw[fill=white] (0.3,0.15)--(0.6,1.1)--(1.2,1.1)--(0.9,0.15)--(0.3,0.15)--cycle;
     \draw (.98,0.76) node[scale=.75, rotate=-12] {$2$};
     \draw (0.45,0.12) node[scale=.75, rotate=-12] {$2$};
     \end{tikzpicture}
     };
  \draw(.75,3+.7) node{
    \begin{tikzpicture}[baseline=0]
     \draw[fill=white] (0.3,0.15)--(0.6,1.1)--(1.2,1.1)--(0.9,0.15)--(0.3,0.15)--cycle;
     \draw (.98,0.76) node[scale=.75, rotate=-12] {$1$};
     \draw (0.45,0.12) node[scale=.75, rotate=-12] {$1$};
     \end{tikzpicture}
     };
  \draw(.75,0) node{
    \begin{tikzpicture}[baseline=0]
     \draw[fill=white] (0.3,0.15)--(0.6,1.1)--(1.2,1.1)--(0.9,0.15)--(0.3,0.15)--cycle;
     \draw (.98,0.76) node[scale=.75, rotate=-12] {$3$};
     \draw (0.45,0.12) node[scale=.75, rotate=-12] {$3$};
     \end{tikzpicture}
     };
  \draw(.75,.35) node{
    \begin{tikzpicture}[baseline=0]
     \draw[fill=white] (0.3,0.15)--(0.6,1.1)--(1.2,1.1)--(0.9,0.15)--(0.3,0.15)--cycle;
     \draw (.98,0.76) node[scale=.75, rotate=-12] {$4$};
     \draw (0.45,0.12) node[scale=.75, rotate=-12] {$4$};
     \end{tikzpicture}
     };
 \end{tikzpicture}
};
\draw(8,0) node[above] {
 \begin{tikzpicture}[baseline=0]
\draw[line width=3, rounded corners] (0,0)--(0,6)--(1.5,6)--(1.5,0)--(0,0)--cycle;
  \draw[line width=3] (0,3)--(1.5,3);
  \draw(.75,3) node{
    \begin{tikzpicture}[baseline=0]
     \draw[fill=white] (0.3,0.15)--(0.6,1.1)--(1.2,1.1)--(0.9,0.15)--(0.3,0.15)--cycle;
     \draw (.98,0.76) node[scale=.75, rotate=-12] {$5$};
     \draw (0.45,0.12) node[scale=.75, rotate=-12] {$5$};
     \end{tikzpicture}
     };
  \draw(.75,3+.35) node{
    \begin{tikzpicture}[baseline=0]
     \draw[fill=white] (0.3,0.15)--(0.6,1.1)--(1.2,1.1)--(0.9,0.15)--(0.3,0.15)--cycle;
     \draw (.98,0.76) node[scale=.75, rotate=-12] {$2$};
     \draw (0.45,0.12) node[scale=.75, rotate=-12] {$2$};
     \end{tikzpicture}
     };
  \draw(.75,3+.7) node{
    \begin{tikzpicture}[baseline=0]
     \draw[fill=white] (0.3,0.15)--(0.6,1.1)--(1.2,1.1)--(0.9,0.15)--(0.3,0.15)--cycle;
     \draw (.98,0.76) node[scale=.75, rotate=-12] {$1$};
     \draw (0.45,0.12) node[scale=.75, rotate=-12] {$1$};
     \end{tikzpicture}
     };
  \draw(.75,0) node{
    \begin{tikzpicture}[baseline=0]
     \draw[fill=white!80!black] (0.3,0.15)--(0.6,1.1)--(1.2,1.1)--(0.9,0.15)--(0.3,0.15)--cycle;
     \draw (.98,0.76) node[scale=.75, rotate=-12] {$6$};
     \draw (0.45,0.12) node[scale=.75, rotate=-12] {$6$};
     \end{tikzpicture}
     };
  \draw(.75,.35) node{
    \begin{tikzpicture}[baseline=0]
     \draw[fill=white] (0.3,0.15)--(0.6,1.1)--(1.2,1.1)--(0.9,0.15)--(0.3,0.15)--cycle;
     \draw (.98,0.76) node[scale=.75, rotate=-12] {$3$};
     \draw (0.45,0.12) node[scale=.75, rotate=-12] {$3$};
     \end{tikzpicture}
     };
  \draw(.75,.7) node{
    \begin{tikzpicture}[baseline=0]
     \draw[fill=white] (0.3,0.15)--(0.6,1.1)--(1.2,1.1)--(0.9,0.15)--(0.3,0.15)--cycle;
     \draw (.98,0.76) node[scale=.75, rotate=-12] {$4$};
     \draw (0.45,0.12) node[scale=.75, rotate=-12] {$4$};
     \end{tikzpicture}
     };
 \end{tikzpicture}
};
\draw(9.5,0) node[above] {
 \begin{tikzpicture}[baseline=0]
\draw[line width=3, rounded corners] (0,0)--(0,6)--(1.5,6)--(1.5,0)--(0,0)--cycle;
  \draw[line width=3] (0,3)--(1.5,3);
  \draw(.75,3) node{
    \begin{tikzpicture}[baseline=0]
     \draw[fill=white] (0.3,0.15)--(0.6,1.1)--(1.2,1.1)--(0.9,0.15)--(0.3,0.15)--cycle;
     \draw (.98,0.76) node[scale=.75, rotate=-12] {$5$};
     \draw (0.45,0.12) node[scale=.75, rotate=-12] {$5$};
     \end{tikzpicture}
     };
  \draw(.75,3+.35) node{
    \begin{tikzpicture}[baseline=0]
     \draw[fill=white] (0.3,0.15)--(0.6,1.1)--(1.2,1.1)--(0.9,0.15)--(0.3,0.15)--cycle;
     \draw (.98,0.76) node[scale=.75, rotate=-12] {$2$};
     \draw (0.45,0.12) node[scale=.75, rotate=-12] {$2$};
     \end{tikzpicture}
     };
  \draw(.75,3+.7) node{
    \begin{tikzpicture}[baseline=0]
     \draw[fill=white] (0.3,0.15)--(0.6,1.1)--(1.2,1.1)--(0.9,0.15)--(0.3,0.15)--cycle;
     \draw (.98,0.76) node[scale=.75, rotate=-12] {$1$};
     \draw (0.45,0.12) node[scale=.75, rotate=-12] {$1$};
     \end{tikzpicture}
     };
  \draw(.75,0) node{
    \begin{tikzpicture}[baseline=0]
     \draw[fill=white] (0.3,0.15)--(0.6,1.1)--(1.2,1.1)--(0.9,0.15)--(0.3,0.15)--cycle;
     \draw (.98,0.76) node[scale=.75, rotate=-12] {$6$};
     \draw (0.45,0.12) node[scale=.75, rotate=-12] {$6$};
     \end{tikzpicture}
     };
  \draw(.75,.35) node{
    \begin{tikzpicture}[baseline=0]
     \draw[fill=white] (0.3,0.15)--(0.6,1.1)--(1.2,1.1)--(0.9,0.15)--(0.3,0.15)--cycle;
     \draw (.98,0.76) node[scale=.75, rotate=-12] {$3$};
     \draw (0.45,0.12) node[scale=.75, rotate=-12] {$3$};
     \end{tikzpicture}
     };
  \draw(.75,.7) node{
    \begin{tikzpicture}[baseline=0]
     \draw[fill=white] (0.3,0.15)--(0.6,1.1)--(1.2,1.1)--(0.9,0.15)--(0.3,0.15)--cycle;
     \draw (.98,0.76) node[scale=.75, rotate=-12] {$4$};
     \draw (0.45,0.12) node[scale=.75, rotate=-12] {$4$};
     \end{tikzpicture}
     };
  \draw(.75,1.05) node{
    \begin{tikzpicture}[baseline=0]
     \draw[fill=white!80!black] (0.3,0.15)--(0.6,1.1)--(1.2,1.1)--(0.9,0.15)--(0.3,0.15)--cycle;
     \draw (.98,0.76) node[scale=.75, rotate=-12] {$7$};
     \draw (0.45,0.12) node[scale=.75, rotate=-12] {$7$};
     \end{tikzpicture}
     };
 \end{tikzpicture}
};
\draw(11,0) node[above] {
 \begin{tikzpicture}[baseline=0]
  \draw[line width=3, rounded corners] (0,0)--(0,6)--(1.5,6)--(1.5,0)--(0,0)--cycle;
  \draw[line width=3] (0,3)--(1.5,3);
  \draw(.75,3) node{
    \begin{tikzpicture}[baseline=0]
     \draw[fill=white] (0.3,0.15)--(0.6,1.1)--(1.2,1.1)--(0.9,0.15)--(0.3,0.15)--cycle;
     \draw (.98,0.76) node[scale=.75, rotate=-12] {$5$};
     \draw (0.45,0.12) node[scale=.75, rotate=-12] {$5$};
     \end{tikzpicture}
     };
  \draw(.75,3+.35) node{
    \begin{tikzpicture}[baseline=0]
     \draw[fill=white] (0.3,0.15)--(0.6,1.1)--(1.2,1.1)--(0.9,0.15)--(0.3,0.15)--cycle;
     \draw (.98,0.76) node[scale=.75, rotate=-12] {$2$};
     \draw (0.45,0.12) node[scale=.75, rotate=-12] {$2$};
     \end{tikzpicture}
     };
  \draw(.75,3+.7) node{
    \begin{tikzpicture}[baseline=0]
     \draw[fill=white] (0.3,0.15)--(0.6,1.1)--(1.2,1.1)--(0.9,0.15)--(0.3,0.15)--cycle;
     \draw (.98,0.76) node[scale=.75, rotate=-12] {$1$};
     \draw (0.45,0.12) node[scale=.75, rotate=-12] {$1$};
     \end{tikzpicture}
     };
  \draw(.75,3+1.05) node{
    \begin{tikzpicture}[baseline=0]
     \draw[fill=white!80!black] (0.3,0.15)--(0.6,1.1)--(1.2,1.1)--(0.9,0.15)--(0.3,0.15)--cycle;
     \draw (.98,0.76) node[scale=.75, rotate=-12] {$8$};
     \draw (0.45,0.12) node[scale=.75, rotate=-12] {$8$};
     \end{tikzpicture}
     };
  \draw(.75,0) node{
    \begin{tikzpicture}[baseline=0]
     \draw[fill=white] (0.3,0.15)--(0.6,1.1)--(1.2,1.1)--(0.9,0.15)--(0.3,0.15)--cycle;
     \draw (.98,0.76) node[scale=.75, rotate=-12] {$6$};
     \draw (0.45,0.12) node[scale=.75, rotate=-12] {$6$};
     \end{tikzpicture}
     };
  \draw(.75,.35) node{
    \begin{tikzpicture}[baseline=0]
     \draw[fill=white] (0.3,0.15)--(0.6,1.1)--(1.2,1.1)--(0.9,0.15)--(0.3,0.15)--cycle;
     \draw (.98,0.76) node[scale=.75, rotate=-12] {$3$};
     \draw (0.45,0.12) node[scale=.75, rotate=-12] {$3$};
     \end{tikzpicture}
     };
  \draw(.75,.7) node{
    \begin{tikzpicture}[baseline=0]
     \draw[fill=white] (0.3,0.15)--(0.6,1.1)--(1.2,1.1)--(0.9,0.15)--(0.3,0.15)--cycle;
     \draw (.98,0.76) node[scale=.75, rotate=-12] {$4$};
     \draw (0.45,0.12) node[scale=.75, rotate=-12] {$4$};
     \end{tikzpicture}
     };
  \draw(.75,1.05) node{
    \begin{tikzpicture}[baseline=0]
     \draw[fill=white] (0.3,0.15)--(0.6,1.1)--(1.2,1.1)--(0.9,0.15)--(0.3,0.15)--cycle;
     \draw (.98,0.76) node[scale=.75, rotate=-12] {$7$};
     \draw (0.45,0.12) node[scale=.75, rotate=-12] {$7$};
     \end{tikzpicture}
     };
 \end{tikzpicture}
};
\draw(12.5,0) node[above] {
 \begin{tikzpicture}[baseline=0]
  \draw[line width=3, rounded corners] (0,0)--(0,6)--(1.5,6)--(1.5,0)--(0,0)--cycle;
  \draw[line width=3] (0,3)--(1.5,3);
  \draw(.75,3) node{
    \begin{tikzpicture}[baseline=0]
     \draw[fill=white] (0.3,0.15)--(0.6,1.1)--(1.2,1.1)--(0.9,0.15)--(0.3,0.15)--cycle;
     \draw (.98,0.76) node[scale=.75, rotate=-12] {$5$};
     \draw (0.45,0.12) node[scale=.75, rotate=-12] {$5$};
     \end{tikzpicture}
     };
  \draw(.75,3+.35) node{
    \begin{tikzpicture}[baseline=0]
     \draw[fill=white] (0.3,0.15)--(0.6,1.1)--(1.2,1.1)--(0.9,0.15)--(0.3,0.15)--cycle;
     \draw (.98,0.76) node[scale=.75, rotate=-12] {$2$};
     \draw (0.45,0.12) node[scale=.75, rotate=-12] {$2$};
     \end{tikzpicture}
     };
  \draw(.75,3+.7) node{
    \begin{tikzpicture}[baseline=0]
     \draw[fill=white] (0.3,0.15)--(0.6,1.1)--(1.2,1.1)--(0.9,0.15)--(0.3,0.15)--cycle;
     \draw (.98,0.76) node[scale=.75, rotate=-12] {$1$};
     \draw (0.45,0.12) node[scale=.75, rotate=-12] {$1$};
     \end{tikzpicture}
     };
  \draw(.75,3+1.05) node{
    \begin{tikzpicture}[baseline=0]
     \draw[fill=white] (0.3,0.15)--(0.6,1.1)--(1.2,1.1)--(0.9,0.15)--(0.3,0.15)--cycle;
     \draw (.98,0.76) node[scale=.75, rotate=-12] {$8$};
     \draw (0.45,0.12) node[scale=.75, rotate=-12] {$8$};
     \end{tikzpicture}
     };
  \draw(.75,3+1.4) node{
    \begin{tikzpicture}[baseline=0]
     \draw[fill=white!80!black] (0.3,0.15)--(0.6,1.1)--(1.2,1.1)--(0.9,0.15)--(0.3,0.15)--cycle;
     \draw (.98,0.76) node[scale=.75, rotate=-12] {$9$};
     \draw (0.45,0.12) node[scale=.75, rotate=-12] {$9$};
     \end{tikzpicture}
     };
  \draw(.75,0) node{
    \begin{tikzpicture}[baseline=0]
     \draw[fill=white] (0.3,0.15)--(0.6,1.1)--(1.2,1.1)--(0.9,0.15)--(0.3,0.15)--cycle;
     \draw (.98,0.76) node[scale=.75, rotate=-12] {$6$};
     \draw (0.45,0.12) node[scale=.75, rotate=-12] {$6$};
     \end{tikzpicture}
     };
  \draw(.75,.35) node{
    \begin{tikzpicture}[baseline=0]
     \draw[fill=white] (0.3,0.15)--(0.6,1.1)--(1.2,1.1)--(0.9,0.15)--(0.3,0.15)--cycle;
     \draw (.98,0.76) node[scale=.75, rotate=-12] {$3$};
     \draw (0.45,0.12) node[scale=.75, rotate=-12] {$3$};
     \end{tikzpicture}
     };
  \draw(.75,.7) node{
    \begin{tikzpicture}[baseline=0]
     \draw[fill=white] (0.3,0.15)--(0.6,1.1)--(1.2,1.1)--(0.9,0.15)--(0.3,0.15)--cycle;
     \draw (.98,0.76) node[scale=.75, rotate=-12] {$4$};
     \draw (0.45,0.12) node[scale=.75, rotate=-12] {$4$};
     \end{tikzpicture}
     };
  \draw(.75,1.05) node{
    \begin{tikzpicture}[baseline=0]
     \draw[fill=white] (0.3,0.15)--(0.6,1.1)--(1.2,1.1)--(0.9,0.15)--(0.3,0.15)--cycle;
     \draw (.98,0.76) node[scale=.75, rotate=-12] {$7$};
     \draw (0.45,0.12) node[scale=.75, rotate=-12] {$7$};
     \end{tikzpicture}
     };
 \end{tikzpicture}
};
\end{tikzpicture}
\bigskip
\caption{A standard shuffle of $n=9$ cards into $m=2$ shelves resulting in the permutation $\pi = 981257436$.}\label{tab:shelf}
\end{table}

\item \textbf{Lazy mode.} This mode is the same as standard mode except that there is a ``Shelf 0'' into which cards can only be placed at the bottom of the shelf. Thus, for each card $i$, we choose either the lazy shelf or one of the $2m$ ordinary (shelf, side) pairs, each with probability $1/(2m+1)$. Each outcome of the machine occurs with probability $1/(2m+1)^n$. For example, with $m=2$ piles, and $n=9$ cards, we might shuffle cards as shown in Table \ref{tab:lazyshelf}. We identify this shuffle with the permutation $\pi = 489125736$. As in prior cases, the same permutation can result from a variety of different outcomes.

\begin{table}
\begin{tikzpicture}
\draw(-.2,-.75) node[left] {Card:};
\draw(-.2,-1.25) node[left] {Shelf:};
\draw(-.2,6.26) node[left] {Shelf 0:};
\draw(-.2,3.75) node[left] {Shelf 1:};
\draw(-.2,1.25) node[left] {Shelf 2:};
\draw (.5,-1.25) node {1t};
\draw (.5,-.75) node {$1$};
\draw (2,-1.25) node {1b};
\draw (2,-.75) node {$2$};
\draw (3.5,-1.25) node {2t};
\draw (3.5,-.75) node {$3$};
\draw (5,-1.25) node {$0$};
\draw (5,-.75) node {$4$};
\draw (6.5,-1.25) node {1b};
\draw (6.5,-.75) node {$5$};
\draw (8,-1.25) node {2b};
\draw (8,-.75) node {$6$};
\draw (9.5,-1.25) node {2t};
\draw (9.5,-.75) node {$7$};
\draw (11,-1.25) node {$0$};
\draw (11,-.75) node {$8$};
\draw (12.5,-1.25) node {1t};
\draw (12.5,-.75) node {$9$};
\draw [line width=4] (-.5,-.25)--(13.5,-.25);
\draw(.5,0) node[above] {
 \begin{tikzpicture}[baseline=0]
  \draw[line width=3, rounded corners] (0,0)--(0,7.5)--(1.5,7.5)--(1.5,0)--(0,0)--cycle;
  \draw[line width=3] (0,2.5)--(1.5,2.5);
  \draw[line width=3] (0,5)--(1.5,5);
  \draw(.75,2.5) node{
    \begin{tikzpicture}[baseline=0]
     \draw[fill=white!80!black] (0.3,0.15)--(0.6,1.1)--(1.2,1.1)--(0.9,0.15)--(0.3,0.15)--cycle;
     \draw (.98,0.76) node[scale=.75, rotate=-12] {$1$};
     \draw (0.45,0.12) node[scale=.75, rotate=-12] {$1$};
     \end{tikzpicture}
     };
 \end{tikzpicture}
};
\draw(2,0) node[above] {
 \begin{tikzpicture}[baseline=0]
  \draw[line width=3, rounded corners] (0,0)--(0,7.5)--(1.5,7.5)--(1.5,0)--(0,0)--cycle;
  \draw[line width=3] (0,2.5)--(1.5,2.5);
  \draw[line width=3] (0,5)--(1.5,5);
  \draw(.75,2.5) node{
    \begin{tikzpicture}[baseline=0]
     \draw[fill=white!80!black] (0.3,0.15)--(0.6,1.1)--(1.2,1.1)--(0.9,0.15)--(0.3,0.15)--cycle;
     \draw (.98,0.76) node[scale=.75, rotate=-12] {$2$};
     \draw (0.45,0.12) node[scale=.75, rotate=-12] {$2$};
     \end{tikzpicture}
     };
  \draw(.75,2.5+.35) node{
    \begin{tikzpicture}[baseline=0]
     \draw[fill=white] (0.3,0.15)--(0.6,1.1)--(1.2,1.1)--(0.9,0.15)--(0.3,0.15)--cycle;
     \draw (.98,0.76) node[scale=.75, rotate=-12] {$1$};
     \draw (0.45,0.12) node[scale=.75, rotate=-12] {$1$};
     \end{tikzpicture}
     };
 \end{tikzpicture}
};
\draw(3.5,0) node[above] {
 \begin{tikzpicture}[baseline=0]
  \draw[line width=3, rounded corners] (0,0)--(0,7.5)--(1.5,7.5)--(1.5,0)--(0,0)--cycle;
  \draw[line width=3] (0,2.5)--(1.5,2.5);
  \draw[line width=3] (0,5)--(1.5,5);
  \draw(.75,2.5) node{
    \begin{tikzpicture}[baseline=0]
     \draw[fill=white] (0.3,0.15)--(0.6,1.1)--(1.2,1.1)--(0.9,0.15)--(0.3,0.15)--cycle;
     \draw (.98,0.76) node[scale=.75, rotate=-12] {$2$};
     \draw (0.45,0.12) node[scale=.75, rotate=-12] {$2$};
     \end{tikzpicture}
     };
  \draw(.75,2.5+.35) node{
    \begin{tikzpicture}[baseline=0]
     \draw[fill=white] (0.3,0.15)--(0.6,1.1)--(1.2,1.1)--(0.9,0.15)--(0.3,0.15)--cycle;
     \draw (.98,0.76) node[scale=.75, rotate=-12] {$1$};
     \draw (0.45,0.12) node[scale=.75, rotate=-12] {$1$};
     \end{tikzpicture}
     };
  \draw(.75,0) node{
    \begin{tikzpicture}[baseline=0]
     \draw[fill=white!80!black] (0.3,0.15)--(0.6,1.1)--(1.2,1.1)--(0.9,0.15)--(0.3,0.15)--cycle;
     \draw (.98,0.76) node[scale=.75, rotate=-12] {$3$};
     \draw (0.45,0.12) node[scale=.75, rotate=-12] {$3$};
     \end{tikzpicture}
     };
 \end{tikzpicture}
};
\draw(5,0) node[above] {
 \begin{tikzpicture}[baseline=0]
  \draw[line width=3, rounded corners] (0,0)--(0,7.5)--(1.5,7.5)--(1.5,0)--(0,0)--cycle;
  \draw[line width=3] (0,2.5)--(1.5,2.5);
  \draw[line width=3] (0,5)--(1.5,5);
  \draw(.75,5) node{
    \begin{tikzpicture}[baseline=0]
     \draw[fill=white!80!black] (0.3,0.15)--(0.6,1.1)--(1.2,1.1)--(0.9,0.15)--(0.3,0.15)--cycle;
     \draw (.98,0.76) node[scale=.75, rotate=-12] {$4$};
     \draw (0.45,0.12) node[scale=.75, rotate=-12] {$4$};
     \end{tikzpicture}
     };
  \draw(.75,2.5) node{
    \begin{tikzpicture}[baseline=0]
     \draw[fill=white] (0.3,0.15)--(0.6,1.1)--(1.2,1.1)--(0.9,0.15)--(0.3,0.15)--cycle;
     \draw (.98,0.76) node[scale=.75, rotate=-12] {$2$};
     \draw (0.45,0.12) node[scale=.75, rotate=-12] {$2$};
     \end{tikzpicture}
     };
  \draw(.75,2.5+.35) node{
    \begin{tikzpicture}[baseline=0]
     \draw[fill=white] (0.3,0.15)--(0.6,1.1)--(1.2,1.1)--(0.9,0.15)--(0.3,0.15)--cycle;
     \draw (.98,0.76) node[scale=.75, rotate=-12] {$1$};
     \draw (0.45,0.12) node[scale=.75, rotate=-12] {$1$};
     \end{tikzpicture}
     };
  \draw(.75,0) node{
    \begin{tikzpicture}[baseline=0]
     \draw[fill=white] (0.3,0.15)--(0.6,1.1)--(1.2,1.1)--(0.9,0.15)--(0.3,0.15)--cycle;
     \draw (.98,0.76) node[scale=.75, rotate=-12] {$3$};
     \draw (0.45,0.12) node[scale=.75, rotate=-12] {$3$};
     \end{tikzpicture}
     };
 \end{tikzpicture}
};
\draw(6.5,0) node[above] {
 \begin{tikzpicture}[baseline=0]
  \draw[line width=3, rounded corners] (0,0)--(0,7.5)--(1.5,7.5)--(1.5,0)--(0,0)--cycle;
  \draw[line width=3] (0,2.5)--(1.5,2.5);
  \draw[line width=3] (0,5)--(1.5,5);
  \draw(.75,5) node{
    \begin{tikzpicture}[baseline=0]
     \draw[fill=white] (0.3,0.15)--(0.6,1.1)--(1.2,1.1)--(0.9,0.15)--(0.3,0.15)--cycle;
     \draw (.98,0.76) node[scale=.75, rotate=-12] {$4$};
     \draw (0.45,0.12) node[scale=.75, rotate=-12] {$4$};
     \end{tikzpicture}
     };
  \draw(.75,2.5) node{
    \begin{tikzpicture}[baseline=0]
     \draw[fill=white!80!black] (0.3,0.15)--(0.6,1.1)--(1.2,1.1)--(0.9,0.15)--(0.3,0.15)--cycle;
     \draw (.98,0.76) node[scale=.75, rotate=-12] {$5$};
     \draw (0.45,0.12) node[scale=.75, rotate=-12] {$5$};
     \end{tikzpicture}
     };
  \draw(.75,2.5+.35) node{
    \begin{tikzpicture}[baseline=0]
     \draw[fill=white] (0.3,0.15)--(0.6,1.1)--(1.2,1.1)--(0.9,0.15)--(0.3,0.15)--cycle;
     \draw (.98,0.76) node[scale=.75, rotate=-12] {$2$};
     \draw (0.45,0.12) node[scale=.75, rotate=-12] {$2$};
     \end{tikzpicture}
     };
  \draw(.75,2.5+.7) node{
    \begin{tikzpicture}[baseline=0]
     \draw[fill=white] (0.3,0.15)--(0.6,1.1)--(1.2,1.1)--(0.9,0.15)--(0.3,0.15)--cycle;
     \draw (.98,0.76) node[scale=.75, rotate=-12] {$1$};
     \draw (0.45,0.12) node[scale=.75, rotate=-12] {$1$};
     \end{tikzpicture}
     };
  \draw(.75,0) node{
    \begin{tikzpicture}[baseline=0]
     \draw[fill=white] (0.3,0.15)--(0.6,1.1)--(1.2,1.1)--(0.9,0.15)--(0.3,0.15)--cycle;
     \draw (.98,0.76) node[scale=.75, rotate=-12] {$3$};
     \draw (0.45,0.12) node[scale=.75, rotate=-12] {$3$};
     \end{tikzpicture}
     };
 \end{tikzpicture}
};
\draw(8,0) node[above] {
 \begin{tikzpicture}[baseline=0]
  \draw[line width=3, rounded corners] (0,0)--(0,7.5)--(1.5,7.5)--(1.5,0)--(0,0)--cycle;
  \draw[line width=3] (0,2.5)--(1.5,2.5);
  \draw[line width=3] (0,5)--(1.5,5);
  \draw(.75,5) node{
    \begin{tikzpicture}[baseline=0]
     \draw[fill=white] (0.3,0.15)--(0.6,1.1)--(1.2,1.1)--(0.9,0.15)--(0.3,0.15)--cycle;
     \draw (.98,0.76) node[scale=.75, rotate=-12] {$4$};
     \draw (0.45,0.12) node[scale=.75, rotate=-12] {$4$};
     \end{tikzpicture}
     };
  \draw(.75,2.5) node{
    \begin{tikzpicture}[baseline=0]
     \draw[fill=white] (0.3,0.15)--(0.6,1.1)--(1.2,1.1)--(0.9,0.15)--(0.3,0.15)--cycle;
     \draw (.98,0.76) node[scale=.75, rotate=-12] {$5$};
     \draw (0.45,0.12) node[scale=.75, rotate=-12] {$5$};
     \end{tikzpicture}
     };
  \draw(.75,2.5+.35) node{
    \begin{tikzpicture}[baseline=0]
     \draw[fill=white] (0.3,0.15)--(0.6,1.1)--(1.2,1.1)--(0.9,0.15)--(0.3,0.15)--cycle;
     \draw (.98,0.76) node[scale=.75, rotate=-12] {$2$};
     \draw (0.45,0.12) node[scale=.75, rotate=-12] {$2$};
     \end{tikzpicture}
     };
  \draw(.75,2.5+.7) node{
    \begin{tikzpicture}[baseline=0]
     \draw[fill=white] (0.3,0.15)--(0.6,1.1)--(1.2,1.1)--(0.9,0.15)--(0.3,0.15)--cycle;
     \draw (.98,0.76) node[scale=.75, rotate=-12] {$1$};
     \draw (0.45,0.12) node[scale=.75, rotate=-12] {$1$};
     \end{tikzpicture}
     };
  \draw(.75,0) node{
    \begin{tikzpicture}[baseline=0]
     \draw[fill=white!80!black] (0.3,0.15)--(0.6,1.1)--(1.2,1.1)--(0.9,0.15)--(0.3,0.15)--cycle;
     \draw (.98,0.76) node[scale=.75, rotate=-12] {$6$};
     \draw (0.45,0.12) node[scale=.75, rotate=-12] {$6$};
     \end{tikzpicture}
     };
  \draw(.75,.35) node{
    \begin{tikzpicture}[baseline=0]
     \draw[fill=white] (0.3,0.15)--(0.6,1.1)--(1.2,1.1)--(0.9,0.15)--(0.3,0.15)--cycle;
     \draw (.98,0.76) node[scale=.75, rotate=-12] {$3$};
     \draw (0.45,0.12) node[scale=.75, rotate=-12] {$3$};
     \end{tikzpicture}
     };
 \end{tikzpicture}
};
\draw(9.5,0) node[above] {
 \begin{tikzpicture}[baseline=0]
  \draw[line width=3, rounded corners] (0,0)--(0,7.5)--(1.5,7.5)--(1.5,0)--(0,0)--cycle;
  \draw[line width=3] (0,2.5)--(1.5,2.5);
  \draw[line width=3] (0,5)--(1.5,5);
  \draw(.75,5) node{
    \begin{tikzpicture}[baseline=0]
     \draw[fill=white] (0.3,0.15)--(0.6,1.1)--(1.2,1.1)--(0.9,0.15)--(0.3,0.15)--cycle;
     \draw (.98,0.76) node[scale=.75, rotate=-12] {$4$};
     \draw (0.45,0.12) node[scale=.75, rotate=-12] {$4$};
     \end{tikzpicture}
     };
  \draw(.75,2.5) node{
    \begin{tikzpicture}[baseline=0]
     \draw[fill=white] (0.3,0.15)--(0.6,1.1)--(1.2,1.1)--(0.9,0.15)--(0.3,0.15)--cycle;
     \draw (.98,0.76) node[scale=.75, rotate=-12] {$5$};
     \draw (0.45,0.12) node[scale=.75, rotate=-12] {$5$};
     \end{tikzpicture}
     };
  \draw(.75,2.5+.35) node{
    \begin{tikzpicture}[baseline=0]
     \draw[fill=white] (0.3,0.15)--(0.6,1.1)--(1.2,1.1)--(0.9,0.15)--(0.3,0.15)--cycle;
     \draw (.98,0.76) node[scale=.75, rotate=-12] {$2$};
     \draw (0.45,0.12) node[scale=.75, rotate=-12] {$2$};
     \end{tikzpicture}
     };
  \draw(.75,2.5+.7) node{
    \begin{tikzpicture}[baseline=0]
     \draw[fill=white] (0.3,0.15)--(0.6,1.1)--(1.2,1.1)--(0.9,0.15)--(0.3,0.15)--cycle;
     \draw (.98,0.76) node[scale=.75, rotate=-12] {$1$};
     \draw (0.45,0.12) node[scale=.75, rotate=-12] {$1$};
     \end{tikzpicture}
     };
  \draw(.75,0) node{
    \begin{tikzpicture}[baseline=0]
     \draw[fill=white] (0.3,0.15)--(0.6,1.1)--(1.2,1.1)--(0.9,0.15)--(0.3,0.15)--cycle;
     \draw (.98,0.76) node[scale=.75, rotate=-12] {$6$};
     \draw (0.45,0.12) node[scale=.75, rotate=-12] {$6$};
     \end{tikzpicture}
     };
  \draw(.75,.35) node{
    \begin{tikzpicture}[baseline=0]
     \draw[fill=white] (0.3,0.15)--(0.6,1.1)--(1.2,1.1)--(0.9,0.15)--(0.3,0.15)--cycle;
     \draw (.98,0.76) node[scale=.75, rotate=-12] {$3$};
     \draw (0.45,0.12) node[scale=.75, rotate=-12] {$3$};
     \end{tikzpicture}
     };
  \draw(.75,.7) node{
    \begin{tikzpicture}[baseline=0]
     \draw[fill=white!80!black] (0.3,0.15)--(0.6,1.1)--(1.2,1.1)--(0.9,0.15)--(0.3,0.15)--cycle;
     \draw (.98,0.76) node[scale=.75, rotate=-12] {$7$};
     \draw (0.45,0.12) node[scale=.75, rotate=-12] {$7$};
     \end{tikzpicture}
     };
 \end{tikzpicture}
};
\draw(11,0) node[above] {
 \begin{tikzpicture}[baseline=0]
  \draw[line width=3, rounded corners] (0,0)--(0,7.5)--(1.5,7.5)--(1.5,0)--(0,0)--cycle;
  \draw[line width=3] (0,2.5)--(1.5,2.5);
  \draw[line width=3] (0,5)--(1.5,5);
  \draw(.75,5) node{
    \begin{tikzpicture}[baseline=0]
     \draw[fill=white!80!black] (0.3,0.15)--(0.6,1.1)--(1.2,1.1)--(0.9,0.15)--(0.3,0.15)--cycle;
     \draw (.98,0.76) node[scale=.75, rotate=-12] {$8$};
     \draw (0.45,0.12) node[scale=.75, rotate=-12] {$8$};
     \end{tikzpicture}
     };
  \draw(.75,5+.35) node{
    \begin{tikzpicture}[baseline=0]
     \draw[fill=white] (0.3,0.15)--(0.6,1.1)--(1.2,1.1)--(0.9,0.15)--(0.3,0.15)--cycle;
     \draw (.98,0.76) node[scale=.75, rotate=-12] {$4$};
     \draw (0.45,0.12) node[scale=.75, rotate=-12] {$4$};
     \end{tikzpicture}
     };
  \draw(.75,2.5) node{
    \begin{tikzpicture}[baseline=0]
     \draw[fill=white] (0.3,0.15)--(0.6,1.1)--(1.2,1.1)--(0.9,0.15)--(0.3,0.15)--cycle;
     \draw (.98,0.76) node[scale=.75, rotate=-12] {$5$};
     \draw (0.45,0.12) node[scale=.75, rotate=-12] {$5$};
     \end{tikzpicture}
     };
  \draw(.75,2.5+.35) node{
    \begin{tikzpicture}[baseline=0]
     \draw[fill=white] (0.3,0.15)--(0.6,1.1)--(1.2,1.1)--(0.9,0.15)--(0.3,0.15)--cycle;
     \draw (.98,0.76) node[scale=.75, rotate=-12] {$2$};
     \draw (0.45,0.12) node[scale=.75, rotate=-12] {$2$};
     \end{tikzpicture}
     };
  \draw(.75,2.5+.7) node{
    \begin{tikzpicture}[baseline=0]
     \draw[fill=white] (0.3,0.15)--(0.6,1.1)--(1.2,1.1)--(0.9,0.15)--(0.3,0.15)--cycle;
     \draw (.98,0.76) node[scale=.75, rotate=-12] {$1$};
     \draw (0.45,0.12) node[scale=.75, rotate=-12] {$1$};
     \end{tikzpicture}
     };
  \draw(.75,0) node{
    \begin{tikzpicture}[baseline=0]
     \draw[fill=white] (0.3,0.15)--(0.6,1.1)--(1.2,1.1)--(0.9,0.15)--(0.3,0.15)--cycle;
     \draw (.98,0.76) node[scale=.75, rotate=-12] {$6$};
     \draw (0.45,0.12) node[scale=.75, rotate=-12] {$6$};
     \end{tikzpicture}
     };
  \draw(.75,.35) node{
    \begin{tikzpicture}[baseline=0]
     \draw[fill=white] (0.3,0.15)--(0.6,1.1)--(1.2,1.1)--(0.9,0.15)--(0.3,0.15)--cycle;
     \draw (.98,0.76) node[scale=.75, rotate=-12] {$3$};
     \draw (0.45,0.12) node[scale=.75, rotate=-12] {$3$};
     \end{tikzpicture}
     };
  \draw(.75,.7) node{
    \begin{tikzpicture}[baseline=0]
     \draw[fill=white] (0.3,0.15)--(0.6,1.1)--(1.2,1.1)--(0.9,0.15)--(0.3,0.15)--cycle;
     \draw (.98,0.76) node[scale=.75, rotate=-12] {$7$};
     \draw (0.45,0.12) node[scale=.75, rotate=-12] {$7$};
     \end{tikzpicture}
     };
 \end{tikzpicture}
};
\draw(12.5,0) node[above] {
 \begin{tikzpicture}[baseline=0]
  \draw[line width=3, rounded corners] (0,0)--(0,7.5)--(1.5,7.5)--(1.5,0)--(0,0)--cycle;
  \draw[line width=3] (0,2.5)--(1.5,2.5);
  \draw[line width=3] (0,5)--(1.5,5);
  \draw(.75,5) node{
    \begin{tikzpicture}[baseline=0]
     \draw[fill=white] (0.3,0.15)--(0.6,1.1)--(1.2,1.1)--(0.9,0.15)--(0.3,0.15)--cycle;
     \draw (.98,0.76) node[scale=.75, rotate=-12] {$8$};
     \draw (0.45,0.12) node[scale=.75, rotate=-12] {$8$};
     \end{tikzpicture}
     };
  \draw(.75,5+.35) node{
    \begin{tikzpicture}[baseline=0]
     \draw[fill=white] (0.3,0.15)--(0.6,1.1)--(1.2,1.1)--(0.9,0.15)--(0.3,0.15)--cycle;
     \draw (.98,0.76) node[scale=.75, rotate=-12] {$4$};
     \draw (0.45,0.12) node[scale=.75, rotate=-12] {$4$};
     \end{tikzpicture}
     };
  \draw(.75,2.5) node{
    \begin{tikzpicture}[baseline=0]
     \draw[fill=white] (0.3,0.15)--(0.6,1.1)--(1.2,1.1)--(0.9,0.15)--(0.3,0.15)--cycle;
     \draw (.98,0.76) node[scale=.75, rotate=-12] {$5$};
     \draw (0.45,0.12) node[scale=.75, rotate=-12] {$5$};
     \end{tikzpicture}
     };
  \draw(.75,2.5+.35) node{
    \begin{tikzpicture}[baseline=0]
     \draw[fill=white] (0.3,0.15)--(0.6,1.1)--(1.2,1.1)--(0.9,0.15)--(0.3,0.15)--cycle;
     \draw (.98,0.76) node[scale=.75, rotate=-12] {$2$};
     \draw (0.45,0.12) node[scale=.75, rotate=-12] {$2$};
     \end{tikzpicture}
     };
  \draw(.75,2.5+.7) node{
    \begin{tikzpicture}[baseline=0]
     \draw[fill=white] (0.3,0.15)--(0.6,1.1)--(1.2,1.1)--(0.9,0.15)--(0.3,0.15)--cycle;
     \draw (.98,0.76) node[scale=.75, rotate=-12] {$1$};
     \draw (0.45,0.12) node[scale=.75, rotate=-12] {$1$};
     \end{tikzpicture}
     };
  \draw(.75,2.5+1.05) node{
    \begin{tikzpicture}[baseline=0]
     \draw[fill=white!80!black] (0.3,0.15)--(0.6,1.1)--(1.2,1.1)--(0.9,0.15)--(0.3,0.15)--cycle;
     \draw (.98,0.76) node[scale=.75, rotate=-12] {$9$};
     \draw (0.45,0.12) node[scale=.75, rotate=-12] {$9$};
     \end{tikzpicture}
     };
  \draw(.75,0) node{
    \begin{tikzpicture}[baseline=0]
     \draw[fill=white] (0.3,0.15)--(0.6,1.1)--(1.2,1.1)--(0.9,0.15)--(0.3,0.15)--cycle;
     \draw (.98,0.76) node[scale=.75, rotate=-12] {$6$};
     \draw (0.45,0.12) node[scale=.75, rotate=-12] {$6$};
     \end{tikzpicture}
     };
  \draw(.75,.35) node{
    \begin{tikzpicture}[baseline=0]
     \draw[fill=white] (0.3,0.15)--(0.6,1.1)--(1.2,1.1)--(0.9,0.15)--(0.3,0.15)--cycle;
     \draw (.98,0.76) node[scale=.75, rotate=-12] {$3$};
     \draw (0.45,0.12) node[scale=.75, rotate=-12] {$3$};
     \end{tikzpicture}
     };
  \draw(.75,.7) node{
    \begin{tikzpicture}[baseline=0]
     \draw[fill=white] (0.3,0.15)--(0.6,1.1)--(1.2,1.1)--(0.9,0.15)--(0.3,0.15)--cycle;
     \draw (.98,0.76) node[scale=.75, rotate=-12] {$7$};
     \draw (0.45,0.12) node[scale=.75, rotate=-12] {$7$};
     \end{tikzpicture}
     };
 \end{tikzpicture}
};
\end{tikzpicture}
\bigskip
\caption{A lazy shuffle of $n=9$ cards into $m=2$ shelves resulting in the permutation $\pi = 489125736$.}\label{tab:lazyshelf}
\end{table}

\end{itemize}

\subsection{Riffle shuffling}

Motivated by \cite[Section 3.1]{DFH}, we will see that our 3 models of shelf shuffling are equivalent to inverse riffle shuffling for 3 models of riffle shuffling. (That is, the probability of a permutation $\pi$ under the shelf shuffling distribution will have the same probability as $\pi^{-1}$ in the corresponding riffle shuffling distribution.) We now describe the three types of riffle shuffling that correspond to our shelf shuffling machine. In the descriptions below, we use the notation $A=(a_1,\ldots,a_m)$ for a \emph{weak composition} of $n$, with $a_i\geq 0$ and $\sum a_i =n$, and we write the multinomial coefficient as
\[
\binom{n}{A} = \binom{n}{a_1,\ldots,a_m} = \frac{n!}{a_1!\cdots a_m!}.
\]

\begin{itemize}
\item \textbf{Riffle shuffle}. This is the classic Gilbert-Shannon-Reeds model of card shuffling, as analyzed in \cite{BD}. Cut the deck into $m$ piles according to the multinomial distribution. The piles have sizes $A=(a_1,\ldots,a_m)$ with probability $\binom{n}{A}/m^n$. To be clear, the first pile contains cards $1,\ldots,a_1$, the second contains cards $a_1+1,\ldots,a_1+a_2$, and so on.

We then ``riffle'' the cards by dropping a card from the bottom of pile $i$ with probability proportional to the size of the pile, until all piles are empty. A straightforward computation shows that this gives the uniform distribution on all $\binom{n}{A}$ interleavings of the piles.  Call $\pi$ the permutation of the cards that results.

We remark that a fixed permutation $\pi$ can result from many different weak compositions. However given a fixed weak composition $A$, there is at most one interleaving of the piles indexed by $A$ that gives $\pi$. If we keep track of the initial pile sizes as well as $\pi$, we see the pair $(A,\pi)$ occurs with probability
\[
\frac{\binom{n}{A}}{m^n}\cdot \frac{1}{\binom{n}{A}} = \frac{1}{m^n}.
\]

\item \textbf{Down-up riffle shuffle}.
Cut the deck into $2m$ piles according to the multinomial distribution. The piles have sizes $A=(b_1, a_1,\ldots, b_m, a_m)$ with probability $\binom{n}{A}/(2m)^n$.
This time we put every other pile \emph{in reverse order}, beginning with the first pile. This indicates the first pile has cards $b_1, b_1-1,\ldots,1$, the second pile has cards $b_1+1,\ldots,b_1+a_1$, the third pile has cards $b_1+a_1+b_2, b_1+a_1+b_2-1,\ldots, b_1+a_1+1$, and so on.

We now riffle the cards as before to give the uniform distribution on all $\binom{n}{A}$ interleavings of the piles.  Call $\pi$ the permutation of the cards that results. If we keep track of the initial pile sizes as well as $\pi$, we see the pair $(A,\pi)$ occurs with probability $1/(2m)^n$.

We note that down-up riffle shuffles were studied in \cite{BD} for $m=1$ and for general $m$ in
\cite{F1}.

\item \textbf{Up-down riffle shuffle}.
This method modifies the down-up riffle shuffle only slightly. First, cut the deck into $2m+1$ piles according to the multinomial distribution. The piles have sizes $A=(a_0,b_1, a_1,\ldots, b_m, a_m)$ with probability $\binom{n}{A}/(2m+1)^n$.
Every other pile is in reverse order, beginning with the second pile. This gives the first pile as cards $1,2,\ldots,a_0$, the second pile as $a_0+b_1, a_0+b_1-1,\ldots,a_0+1$, the third pile as cards $a_0+b_1+1, \ldots, a_0+b_1+a_1$, and so on.

Again we riffle the cards to give the uniform distribution on all $\binom{n}{A}$ interleavings of the piles.  Call $\pi$ the permutation of the cards that results. If we keep track of the initial pile sizes as well as $\pi$, we see the pair $(A,\pi)$ occurs with probability $1/(2m+1)^n$.

We note that up-down riffle shuffles were studied in \cite{BB} for $m=1$ and for general $m$ in
\cite{F1}.
\end{itemize}

\begin{rem}[``Outcomes'']
In the descriptions of all shelf shuffling and riffle shuffling varieties described above, we have used the word ``outcome'' rather loosely. In the shelf shuffling examples, an ``outcome'' refers to the sequence of card placements, not only the permutation of the cards at the end of the sequence. In the case of the riffle shuffles, an ``outcome'' refers to the weak composition-permutation pair $(A,\pi)$. We point out that the outcomes of shelf shuffling can also be encoded with weak composition-permutation pairs, by recording the number of cards placed on the top and bottom of each shelf throughout the sequence of events. For example, the sequence for the lazy shuffle shown in Table \ref{tab:lazyshelf} is (1t, 1b, 2t, 0, 1b, 2b, 2t, 0, 1t). There are two occurrences of ``$0$'' in the sequence, two occurrences of ``1t,'' two occurrences of ``1b,'' two occurrences of ``2t,'' and one occurrence of ``2b.'' We can encode this information in the weak composition $(2, 2, 2, 2, 1)$, and together with the permutation $\pi=489125736$, we have all the information we need to recover the sequence of card placements. (Indeed, we easily deduce which cards ended up on which shelf, and the card labels tell us which card entered a particular shelf most recently.) From this point forward, when we use the word ``outcome'' in reference to a shuffle, it is best to think of the weak composition-permutation pair.
\end{rem}

\subsection{$P$-partitions}

Now we turn our attention to $P$-partitions, but first we discuss background for partially ordered sets. See \cite[Chapter 4]{EC1} for more.

Throughout, we fix a positive integer $n$ and let $P$ denote a partial ordering of the set $[n]=\{1,2,\ldots,n\}$. We write ``$<_P$'' for the order relation on $P$, i.e., if $i$ is below $j$ in $P$, we say $i$ and $j$ are \emph{comparable} and write $i <_P j$ or $j >_P i$. If neither $i<_P j$ nor $i >_P j$, we say $i$ and $j$ are \emph{incomparable}. A comparable pair $i<_P j$ is \emph{naturally labeled} if $i <_{\NN} j$ as well. Otherwise, the pair is \emph{unnaturally labeled}.

A \emph{chain} is a poset in which any two elements are comparable. The $n$ element \emph{antichain}, denoted $[n]$, is the poset with no relations. We readily identify chains with permutations, via
\[
 \pi(1) <_{\pi} \pi(2) <_{\pi} \cdots <_{\pi} \pi(n),
\]
whenever $P=\pi$ is a chain. To say that $i <_{\pi} j$ is equivalent to saying that $\pi^{-1}(i) <_{\NN} \pi^{-1}(j)$.

We say that $Q$ \emph{refines} $P$ if every relation in $P$ is a relation in $Q$. That is, $i <_P j$ implies $i <_Q j$. In this setting, chains are maximally refined posets. We define the set of \emph{linear extensions} of $P$ to be the set of chains (permutations) $\pi$ such that $\pi$ refines $P$:
\[
 \LL(P) = \{ \pi \in S_n : i <_P j \Rightarrow  i <_{\pi} j \}.
\]

In the definition of $P$-partitions, the set of integers $\ZZ = \{ \ldots, \bar 2, \bar 1, 0, 1, 2, \ldots \}$ is given the ordering
\[
 0 <_{\ZZ} \bar 1 <_{\ZZ} 1 <_{\ZZ} \bar 2 <_{\ZZ} 2 < \cdots,
\]
where we write $\bar i$ instead of $-i$ to help avoid confusion with respect to the usual integer ordering. We define symbols ``$\lp$'' and ``$\len$'' as follows:
\begin{align*}
a \lp b &\Longleftrightarrow a <_{\ZZ} b \mbox{ or } a=b \in \{0, 1, 2, 3, \ldots \},\\
a \len b &\Longleftrightarrow a <_{\ZZ} b \mbox{ or } a=b \in \{\bar 1, \bar 2, \bar 3, \ldots \}.
\end{align*}

\begin{defn}[$P$-partition]
A \emph{$P$-partition} is an order preserving function $f: P \to \ZZ$ such that for $i <_P j$:
\begin{itemize}
\item $f(i) \lp f(j)$ if $i <_{\NN} j$,
\item $f(i) \len f(j)$ if $i >_{\NN} j$.
\end{itemize}
\end{defn}

In other words, the values of a $P$-partition on a naturally labeled pair are only allowed to agree on nonbarred values, while the values on an unnaturally labeled pair can only agree on barred values. We denote the set of $P$-partitions by $\A(P)$.

We consider three subsets of $P$-partitions, characterized by restrictions on the image of $f$.
\begin{itemize}
\item \textbf{Positive $P$-partitions.} A $P$-partition $f$ whose image is in $\NN$ is equivalent to the order-preserving version of Stanley's original definition of a $P$-partition \cite[Chapter 4]{EC1}, i.e., $i<_P j$ implies $f(i)\leq_{\NN} f(j)$ with $f(i)<f(j)$ if $i >_{\NN} j$ is an unnaturally labeled pair. We denote this subset of $P$-partitions as follows:
\[
 \A^+(P) = \{ f \in \A(P) : f(i) \in \NN \mbox{ for all } i \in P \}.
\]
This set will help to encode strict shelf shuffling and classic riffle shuffling.

\item \textbf{Nonzero $P$-partitions.} The $P$-partitions whose image does not contain $0$ are precisely Stembridge's enriched $P$-partitions \cite{Stem}. We denote this set as:
\[
 \A^*(P) = \{ f \in \A(P) : f(i) \neq 0 \mbox{ for all } i \in P \}.
\]
This set will help to encode standard shelf shuffling and down-up riffle shuffling.

\item \textbf{All $P$-partitions.} Without any restrictions, this is precisely the definition of left enriched $P$-partitions given by Petersen \cite{Pet}. This set will help to encode lazy shelf shuffling and up-down riffle shuffling.
\end{itemize}

For example, if $n=3$ and $P$ is the poset with $1 <_P 2$ and $3 <_P 2$, it has linear extensions $\LL(P)=\{132, 312\}$. We can draw the poset and its extensions with \emph{Hasse diagrams} as indicated here:
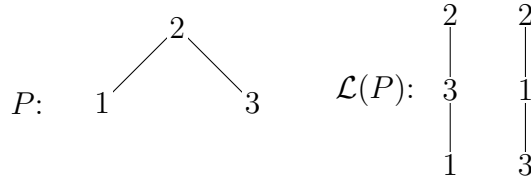
\[
 \begin{tikzpicture}[baseline=0]
  \draw (0,1) node {$P$:};
  \draw (1,1) node[fill=white,inner sep=1] {$1$}-- (2,2) node[fill=white,inner sep=1] {$2$} -- (3,1) node[fill=white,inner sep=1] {$3$};
 \end{tikzpicture}
 \qquad
 \begin{tikzpicture}
  \draw (0,1) node {$\LL(P)$:};
  \draw (1,0) node[fill=white,inner sep=1] {$1$}-- (1,1) node[fill=white,inner sep=1] {$3$} -- (1,2) node[fill=white,inner sep=1] {$2$};
  \draw (2,0) node[fill=white,inner sep=1] {$3$}-- (2,1) node[fill=white,inner sep=1] {$1$} -- (2,2) node[fill=white,inner sep=1] {$2$};
 \end{tikzpicture}
\]
In this case, every $P$-partition $f$ must satisfy
\[
 f(1)\lp f(2) \gn f(3),
\]
so
\[
 \A(P) = \{ (a_1, a_2, a_3)\in  \ZZ^3 : a_1 \lp a_2 \gn a_3 \}.
\]
We can write this as a disjoint union:
\[
 \A(P) = \{ a_1 \lp a_3 \len a_2 \} \cup \{ a_3 \len a_1 \lp a_2 \}.
\]
But each of these smaller sets can viewed as the $P$-partitions for a chain:
\[
 \A(P) = \A(132) \cup \A(312).
\]

By induction on the number of incomparable pairs in a general poset $P$, we can see that, in general, the set of all $P$-partitions is the disjoint union of the $P$-partitions for its chains.

\begin{thm}[\cite{EC1}, Lemma 4.5.3]\label{thm:ftpp}
The set of $P$-partitions is the disjoint union of the $\pi$-partitions of its linear extensions:
\[
 \A(P) = \bigcup_{\pi \in \LL(P)} \A(\pi).
\]
\end{thm}

The antichain $P=[n]$ is a special case worth considering here, since every $P$-partition of $n$ elements is a $P$-partition for the antichain. Moreover since
\[
 \A([n]) = \bigcup_{\pi \in S_n} \A(\pi),
\]
this implies that each $P$-partition $f$ belongs to just one subset $\A(\pi)$. This determines a unique permutation $\pi=\pi(f)$, which we call the \emph{sorting permutation} for $f$.\footnote{Stanley says  $f$ is \emph{$\pi$-compatible} in this situation \cite[Section 7.19]{EC2}. This perspective emphasizes $f$ relative to $\pi$, whereas our terminology emphasizes $\pi$ relative to $f$.}

\begin{defn}[Sorting permutation of a $P$-partition]\label{def:permf}
Each $P$-partition $f$ determines a unique permutation $\pi=\pi(f)$. We define $\pi$ from $f$ via sorting the pairs $(i,f(i))$ according to:
\begin{itemize}
 \item if $f(i) < f(j)$, then $i <_{\pi} j$,
 \item if $i<_{\NN} j$ and $f(i)=f(j) \in \{0,1,2,\ldots\}$, then $i <_{\pi} j$,
 \item if $i<_{\NN} j$ and $f(i)=f(j) \in \{\bar 1,\bar 2,\bar 3,\ldots\}$, then $j <_{\pi} i$.
\end{itemize}
\end{defn}

We remark that this definition for a sorting permutation makes sense for nonzero and positive $P$-partitions as well.

For example, suppose $n=9$ and we have the following $P$-partition $f$, written in two-line notation with $f(i)$ below $i$:
\begin{equation}\label{eq:fex}
 f = \left( \begin{array}{rrrrrrrrr}
       1 & 2 & 3 & 4 & 5 & 6 & 7 & 8 & 9 \\
       \bar 1 & 0 & 0 & \bar 2 & \bar 1 & 1 & 0 & 2 & 2
       \end{array} \right).
\end{equation}
The image multiset of $f$ is $\{0,0,0,\bar 1,\bar 1,1,\bar 2,2,2\}$, which we can denote $\{0^3, \bar 1^2, 1^1, \bar 2^1, 2^2\}$ for brevity. Since $f(2)=f(3)=f(7)=0$, we know $2$, $3$, and $7$ must be the first three entries of $\pi$ and since the image is 0, these must be in their natural order. Thus, $\pi(1)\pi(2)\pi(3) = 237$. Similarly, since $f(1)=f(5)=\bar 1$ is the next biggest value of $f$, we know the next two entries of $\pi$ are $1$ and $5$. Since the image of these two elements is $\bar 1$, they must appear in reverse order, i.e., $\pi(4)\pi(5)=51$. Continuing in this way, we can deduce the values of each entry of $\pi=\pi(f) = 237516489$.

Thinking in terms of the two-line notation itself, we simply sort the array from left to right according to the bottom row. When there are ties, we sort in increasing order on the top row for $f(i)$ nonnegative, and in decreasing order for $f(i)$ barred:
\[
 f\circ \pi = \left( \begin{array}{rrrrrrrrr}
       2 & 3 & 7 & 5 & 1 & 6 & 4 & 8 & 9 \\
       0 & 0 & 0 & \bar1 & \bar1 & 1 & \bar2 & 2 & 2
       \end{array} \right).
\]
We see $\pi(f) = 237516489$ in the top line of the sorted array.

\subsection{$P$-partitions encode shuffles}

The connection between $P$-partitions and shelf shufflers comes from bounding $P$-partitions. That is, let $\A(P;m)$ denote the set of $P$-partitions with absolute value bounded by $m$, i.e.,
\[
 \A(P;m) =\{ f \in \A(P): |f(i)| \leq m \}.
\] Here we use the notation that $|\bar j|=j$.

When $P=[n]$ is an antichain, we see each $f \in \A([n];m)$ is merely a record of exactly what the shelf shuffler did with each card while in lazy mode. The following proposition should be roughly self-evident.

\begin{prop}[Shelf shuffling and $P$-partitions]\label{prp:equiv}
There is a bijection between $\A([n];m)$ and the set of outcomes of an $m$-shelf shuffler in lazy mode. Namely, we place cards one at a time, from $i=1, 2,\ldots, n$. We place card $i$ on shelf $|f(i)|$. If $f(i)$ is barred, the card is placed on top of the cards on the shelf, while if $f(i)$ is unbarred, the card is placed below the cards already on the shelf. The ordering of the cards after shuffling is given by $\pi(f)$.
\end{prop}

For example, the $P$-partition in \eqref{eq:fex} corresponds to the shelf shuffler taking card 1 and placing it on top of shelf 1, card 2 on the bottom of shelf 0, card 3 on the bottom of shelf 0, and so on.

This correspondence between shuffles and $P$-partitions restricts to other modes in the obvious way. If we want to encode the strict shuffle mode, we use the set of positive $P$-partitions for the antichain,
\[
 \A^+([n];m) =\{ f \in \A^+([n]): |f(i)| \leq m \},
\]
while if we want to encode the standard shuffle mode, we use nonzero $P$-partitions,
\[
 \A^*([n];m) =\{ f \in \A^*([n]): |f(i)| \leq m \}.
\]

\begin{rem}\label{rem:topvbot}
The shelf shufflers in this paper drop cards on shelves one at a time from the top of the deck, whereas in \cite{DFH}, the cards are dealt into shelves one at a time from the bottom of the deck. The correspondence between $P$-partitions and shelf shuffling in Proposition \ref{prp:equiv} can be modified to match the shuffling mechanism from \cite{DFH} as follows. Dealing from the bottom means that we place cards $i=n,\ldots,2,1$.

The $P$-partition example from \eqref{eq:fex},
\[
 f = \left( \begin{array}{rrrrrrrrr}
       1 & 2 & 3 & 4 & 5 & 6 & 7 & 8 & 9 \\
       \bar1 & 0 & 0 & \bar2 & \bar1 & 1 & 0 & 2 & 2
       \end{array} \right),
\]
is now interpreted as ``place card 9 on the bottom of shelf 2, then place card 8 on the bottom of shelf 2, place card 7 on the bottom of shelf 0,'' and so on. We end up sorting the array as
\[
\left( \begin{array}{rrrrrrrrr}
       7 & 3 & 2 & 1 & 5 & 6 & 4 & 9 & 8 \\
       0 & 0 & 0 & \bar1 & \bar1 & 1 & \bar2 & 2 & 2
       \end{array} \right),
\]
so the permutation from the bottom-dealing shelf-shuffler gives permutation $732156498$ rather than the permutation $\pi(f)=237516489$ we found previously. Cards on each shelf alternately increase then decrease with the bottom-dealing mechanism, rather than decreasing then increasing with the top-dealing mechanism. (And the bottom-dealing mechanism has the has cards in decreasing order on shelf 0, rather than increasing order.) We choose the top-dealing mechanism for convenience.

From a statistical standpoint, i.e., the convergence of shuffling to uniformity, it should be clear the difference is trivial. Combinatorially, the effect of this choice will ultimately be for us to explain probabilistic results in terms of \emph{peaks} rather than \emph{valleys} of permutations. The translation between peaks and valleys is discussed in Section 3 of \cite{DFH}, particularly in the discussion around Theorem 3.2 and the proof Theorem 3.1.
\end{rem}

We have an analogous correspondence between $P$-partitions and our three flavors of riffle shuffles, but it is less immediately obvious. To explain the idea for up-down riffle shuffles, suppose \[A = (a_0, b_1, a_1, \ldots, b_m, a_m),\] is a weak composition of $n$. Define the poset $P_A$ to be the union of chains that correspond to the piles formed in the process of an up-down riffle shuffle, with the first $a_0$ cards in increasing order, the next $b_1$ cards in decreasing order, and so on. For example, if $A=(3,4,3,4,2)$, we have $P_A$ is the disjoint union of the five chains below (recall we read the chains from the bottom up):
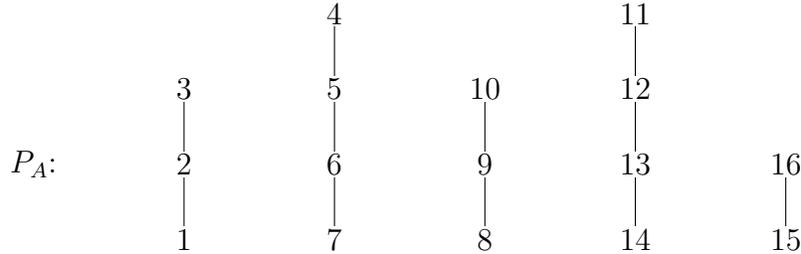
\[
 \begin{tikzpicture}
 \draw (0,1) node {$P_A$:};
 \draw (2,0) node[fill=white, inner sep=1] {$1$}--(2,1) node[fill=white, inner sep=1] {$2$} -- (2,2) node[fill=white, inner sep=1] {$3$};
 \draw (4,0) node[fill=white, inner sep=1] {$7$} --(4,1) node[fill=white, inner sep=1] {$6$} -- (4,2) node[fill=white, inner sep=1] {$5$} -- (4,3) node[fill=white, inner sep=1] {$4$};
 \draw (6,0) node[fill=white, inner sep=1] {$8$} --(6,1) node[fill=white, inner sep=1] {$9$} -- (6,2) node[fill=white, inner sep=1] {$10$};
 \draw (8,0) node[fill=white, inner sep=1] {$14$} --(8,1) node[fill=white, inner sep=1] {$13$} -- (8,2) node[fill=white, inner sep=1] {$12$} -- (8,3) node[fill=white, inner sep=1] {$11$};
 \draw (10,0) node[fill=white, inner sep=1] {$15$} --(10,1) node[fill=white, inner sep=1] {$16$};
 \end{tikzpicture}
\]
Any linear extension $\sigma \in \LL(P_A)$ corresponds precisely to one of the interleavings of the stacks of cards in the riffle shuffle. Moreover, we have $1 <_{\sigma} 2 <_{\sigma} 3$, as well as $4 >_{\sigma} 5 >_{\sigma} 6 >_{\sigma} 7$, and so on, or equivalently:
\[
 \sigma^{-1}(1)<\sigma^{-1}(2)<\sigma^{-1}(3), \quad \sigma^{-1}(4)>\sigma^{-1}(5)>\sigma^{-1}(6)>\sigma^{-1}(7), \ldots.
\]
But this means $\sigma^{-1}$ is the sorting permutation for the $P_A$-partition $f$ such that {\tiny
\[
 f \circ \sigma^{-1} = \left( \begin{array}{rrrrrrrrrrrrrrrr}
       \sigma^{-1}(1) & \sigma^{-1}(2) & \sigma^{-1}(3) & \sigma^{-1}(4) & \sigma^{-1}(5) & \sigma^{-1}(6) & \sigma^{-1}(7) & \cdots & \sigma^{-1}(14) & \sigma^{-1}(15) &\sigma^{-1}(16) \\
       0 & 0 & 0 & \bar1 & \bar1 & \bar1 & \bar1 & \cdots & \bar2 & 2 & 2
       \end{array} \right),
\]
}
with image multiset
\[
 \{ 0,0,0,\bar1,\bar1,\bar1,\bar1,1,1,1,\bar2,\bar2,\bar2,\bar2,2,2\} = \{0^3, \bar1^4, 1^3, \bar2^4, 2^2\}.
\]

To summarize, we are saying that the weak composition-permutation pair $(A,\sigma)$, with $\sigma \in \mathcal{L}(P_A)$, is an outcome of an up-down riffle shuffle, but also, the pair $(A,\sigma)$ corresponds to a unique $P_A$-partition with sorting permutation $\sigma^{-1}$. We summarize this idea in the following proposition.

\begin{prop}[Up-down riffle shuffling and $P$-partitions]\label{prp:equiv2}
There is a bijection between the set of outcomes of up-down $m$-riffle shuffling and $\A([n];m)$. Namely, if $(A,\sigma)$ is an outcome of the shuffle with $A=(a_0,b_1,a_1,\ldots,b_m,a_m)$, then it corresponds to that $f \in \A(P_A;m) \subseteq \A([n];m)$ such that the image of $f$ is $\{0^{a_0}, \bar1^{b_1}, 1^{a_1},\ldots,\overline{m}^{b_m}, m^{a_m}\}$ and $\sigma^{-1}=\pi(f)$ is the sorting permutation of $f$.
\end{prop}

We can modify Proposition \ref{prp:equiv2} to show outcomes of ordinary $m$-riffle shuffling are in bijection with $\A^+([n];m)$ and outcomes of $m$-down-up riffle shuffling correspond to elements of $\A^*([n];m)$.

\section{Shuffling probabilities from $P$-partitions}\label{sec:enum}

In this section we will survey some enumerative results for $P$-partitions and use Propositions \ref{prp:equiv} and \ref{prp:equiv2} to translate them into probabilistic results for shuffling.

\subsection{Enumerative results for $P$-partitions}

The \emph{order polynomial} for $P$, denoted $\op_P(m)$, counts the number of $P$-partitions bounded by $m$, i.e.,
\[
 \op_P(m) = |\A(P;m)| = |\{ f \in \A(P) : |f(i)| \leq m \}|.
\]
We similarly define $\op^*_P(m) = |\A^*(P;m)|$ and $\op^+_P(m)=|\A^+(P;m)|$.

An immediate corollary of Theorem \ref{thm:ftpp} is that order polynomials are sums of order polynomials for linear extensions.\footnote{In fact, this corollary, along with analysis of the case of a chain, gives a simple way to prove that order polynomials are actually polynomials.}

\begin{cor}\label{cor:oplinearex}
The order polynomial for a poset $P$ is the sum of the order polynomials for its linear extensions:
\[
 \op_P(m) = \sum_{\pi \in \LL(P)} \op_{\pi}(m),
\]
and similarly for $\op_P^*(m)$ and $\op_P^+(m)$.
\end{cor}

Order polynomials for antichains are easy enough, since there are no relations to worry about.

\begin{obs}[Antichain order polynomials]\label{obs:antichainop}
Let $[n]$ denote the antichain on $n$ elements. We have, for any $m\geq 0$,
\begin{align*}
 \op_{[n]}(m) &= (2m+1)^n, \\
 \op_{[n]}^*(m) &= (2m)^n, \\
 \op_{[n]}^+(m) &= m^n.
\end{align*}
\end{obs}

The other extreme situation is the case of chains. To understand enumerative properties for order polynomials of chains, we need to discuss permutation statistics. A \emph{descent} of a permutation $\pi$ is an index $i$ such that $\pi(i) > \pi(i+1)$. We let $\des(\pi)$ denote the number of descents of $\pi$. A \emph{peak} of a permutation is an index $i$ such that $\pi(i-1) < \pi(i) > \pi(i+1)$, i.e., a descent preceded by a non-descent. The number of peaks is denoted by $\pk(\pi)$. A \emph{left peak} is a peak of the permutation $\pi$ augmented by $\pi(0)=0$. In other words, a left peak is a peak or a descent in position $1$. We let $\lpk(\pi)$ denote the number of left peaks. We have $\lpk(\pi) = \pk(\pi)$ if $\pi(1) < \pi(2)$ and $\lpk(\pi) = \pk(\pi)+1$ if $\pi(1) > \pi(2)$. We can now give relatively simple expressions for the order polynomials of chains.

\begin{prop}[Order polynomials for chains]\label{prp:chains}
Let $\pi$ denote a chain on $[n]$, i.e., a permutation in $S_n$. We have the following expressions for its order polynomials:
\begin{align}
 \sum_{m\geq 0} \op_{\pi}(m)t^m &= \frac{(4t)^{\lpk(\pi)}(1+t)^{n-2\lpk(\pi)}}{(1-t)^{n+1}}, \label{eq:opl} \\
 \sum_{m \geq 0} \op_{\pi}^*(m) t^m &= \frac{(4t)^{\pk(\pi)+1}(1+t)^{n-2\pk(\pi)-1}}{2(1-t)^{n+1}}, \label{eq:opp}\\
 \sum_{m \geq 0} \op_{\pi}^+(m) t^m &= \frac{ t^{\des(\pi)+1}}{(1-t)^{n+1}}.\label{eq:op}
\end{align}
Equivalently,
\begin{align*}
 \op_{\pi}(m) &= 4^{\lpk(\pi)}\sum_{a\geq 0} \binom{n+m-a}{n}\binom{n-2\lpk(\pi)}{a-\lpk(\pi)}, \\
 \op_{\pi}^*(m) &= 2\cdot 4^{\pk(\pi)}\sum_{a\geq 0} \binom{n-1+m-a}{n}\binom{n-1-2\pk(\pi)}{a-\pk(\pi)},
\mbox{ and}\\
 \op_{\pi}^+(m) &= \binom{n-1+m-\des(\pi)}{n}.
\end{align*}
\end{prop}

Equation \eqref{eq:op} is proved in \cite[Theorem 4.5.14]{EC1}, Equation \eqref{eq:opp} is from \cite[Theorem 4.1]{Stem}, and Equation \eqref{eq:opl} is from \cite[Theorem 4.6]{Pet}.

The proof of Equation \eqref{eq:op} is elementary. The formula for $\op_{\pi}^+(m)$ follows from recognizing that $P$-partitions for a permutation (i.e., a chain) are weakly increasing sequences of integers in an interval:
\[
1\leq f(\pi(1)) \leq f(\pi(2)) \leq \cdots \leq f(\pi(n))\leq m,
\]
with $f(\pi(i)) < f(\pi(i+1))$ if and only if $\pi(i) > \pi(i+1)$, i.e., if $i \in \Des(\pi)$. If there are $k$ strict inequalities, the number of such integer sequences is $\binom{m+(n-1-k)}{n}$.

For example, the integer sequences $(a_1,\ldots,a_7)$ satisfying
\[
 1\leq a_1 \leq a_2 < a_3 \leq a_4 < a_5 < a_6 \leq a_7 \leq m,
\]
also satisfy the inequalities
\[
 1\leq a_1 < (a_2 +1) < (a_3+1) < (a_4+2) < (a_5+2) < (a_6+2) < (a_7 + 3) \leq m+3,
\]
for which the number of solutions is $\binom{m+(6-3)}{7}$.

The proofs of Equations \eqref{eq:opp} and \eqref{eq:opl} are more subtle, but essentially follow from the general idea that the set of enriched $\pi$-partitions for a permutation $\pi$ with peak set $J$ corresponds to a union of sets of ordinary $\tau$-partitions for permutations $\tau$ with descent set containing $J$. See \cite[Section 4]{Pet} for full details.

The theory of $P$-partitions yields the following identities as well, which are useful for explaining repeated shuffles.

\begin{prop}\label{prp:decomp}
For any integers $k$ and $l$ and any permutation $\pi \in S_n$,
\begin{align}
 \op_{\pi}(2kl+k+l) &= \sum_{\sigma\tau=\pi} \op_{\sigma}(k)\op_{\tau}(l),\label{eq:decomp1} \\
 \op^*_{\pi}(2kl) &=\sum_{\sigma\tau=\pi} \op^*_{\sigma}(k)\op^*_{\tau}(l), \label{eq:decomp2}\\
 \op^+_{\pi}(kl) &=\sum_{\sigma\tau=\pi} \op^+_{\sigma}(k)\op^+_{\tau}(l). \label{eq:decomp3}
\end{align}
\end{prop}

We remark that Equation \eqref{eq:permsum} from the introduction is a restatement of identity \eqref{eq:decomp3} for the positive order polynomials $\op_{\pi}^+$.

Each of these identities follows from a similar paradigm of decomposing a \emph{bipartite $P$-partition} $\mathbf{f}: \pi \to S\times T$, where $S$ and $T$ are appropriately chosen totally ordered sets (e.g., $S\times T = [k]\times [l]$ in the simplest case) and $S\times T$ is given a linear ordering (e.g., lexicographic ordering in the simplest case). The image of such a function is a collection of pairs
\[
\mathbf{f}(\pi) = \{ (f_1(1), f_2(1)), (f_1(2),f_2(2)), \ldots, (f_1(n), f_2(n))\},
\]
which can be re-interpreted as a pair of $P$-partitions $f_1 : \sigma \to S$ and $f_2: \tau \to T$, such that $\sigma\tau = \pi$. More details can be found in \cite{Pet}, which proves each of \eqref{eq:decomp1}, \eqref{eq:decomp2}, and \eqref{eq:decomp3}. It should be noted, though, that the case for $\op^+_{\pi}$ (which provides the motivation for the other cases) is found in earlier work of Gessel \cite{Ge}.

When desired, we can drop the permutation from the notation for order polynomials and write only the statistic. That is, fix $n$ and let
\begin{align*}
 \op(n,k;m) &= \op_{\pi}(m) \mbox{ for some $\pi \in S_n$ with $\lpk(\pi)=k$,} \\
 \op^*(n,k;m) &= \op^*_{\pi}(m) \mbox{ for some $\pi \in S_n$ with $\pk(\pi)=k$,
and}\\
 \op^+(n,k;m) &= \op^+_{\pi}(m) \mbox{ for some $\pi \in S_n$ with $\des(\pi)=k$.}
\end{align*}

The following lemma will be useful for some of our convergence estimates later on. For
strict shelf shufflers it is obvious from the explicit formula for $\op^+(n,k;m)$. For
standard shelf shufflers it was given a probabilistic proof in \cite{DFH}. We believe the
result to be new for lazy shelf shufflers, and the proof method to be new for all three
cases. We are able to extend the probabilistic proof of \cite{DFH} to the lazy setting,
but we believe the proof we give here to be more conceptual.

\begin{lemma}[Monotonicity Lemma]\label{lem:monotone}
Order polynomials are monotone decreasing in their respective statistical indices, i.e., for any $m, n,$ and $k$ we have
\begin{align}
\op(n,k; m) &\geq \op(n,k+1;m), \\
\op^*(n,k;m) &\geq \op^*(n,k+1;m),\\
 \op^+(n,k;m) &\geq \op^+(n,k+1;m).
\end{align}
\end{lemma}

\begin{proof}
The argument in each case is to choose a canonical permutation $\pi$ with statistic $k+1$, another permutation $\pi'$ with statistic $k$, and construct an injection from $\A(\pi) \to \A(\pi')$. If $k$ is so large that no such permutation $\pi$ exists, then the order polynomial equals zero and the inequality holds trivially.

We will handle $\op(n,k;m)$ in detail. The arguments for $\op^+$ and $\op^*$ are similar.

Fix $n$, fix $k+1\leq n/2$ and let $\pi \in S_n$ be the permutation that swaps $2i$ and $2i-1$, for each $i=1,\ldots,k+1$. In one-line notation,
\[
\pi = 214365\cdots (2k)(2k-1)(2k+2)(2k+1)(2k+3)\cdots n.
\]
Let $\pi'$ be similar, but with the $(k+1)$st pair unswapped:
\[
\pi' = 214365\cdots (2k)(2k-1)(2k+1)(2k+2)\cdots n.
\]
By construction, $\lpk(\pi) = k+1$ and $\lpk(\pi')=k$.

The condition for a function $f$ in $\A(\pi)$ is:
\[
 f(2)\len f(1) \lp \cdots \len f(2k-1) \lp f(2k+2) \len f(2k+1) \lp f(2k+3) \lp \cdots \lp f(n).
\]
And the condition for a function $g$ in $\A(\pi')$ is:
\begin{equation}\label{eq:g}
 g(2)\len g(1) \lp \cdots \len g(2k-1) \lp g(2k+1) \lp g(2k+2) \lp g(2k+3) \lp \cdots \lp g(n).
\end{equation}

Now let $f \in \A(\pi)$. There are two cases to consider: either $f(2k+2)<f(2k+1)$ or $f(2k+2)=f(2k+1)\in \{ \bar 1, \bar 2, \ldots\}$.

On the one hand, suppose $f(2k+2)<f(2k+1)$. To get a $P$-partition for $\pi'$ we define $g=f'$ by:
\begin{itemize}
\item $f'(2k+1)=f(2k+2)$,
\item $f'(2k+2)=f(2k+1)$, and
\item $f'(i) = f(i)$ otherwise.
\end{itemize}

On the other hand, suppose $f(2k+2)=f(2k+1) = \overline{a}$. Now we define $g=f'$ by:
\begin{itemize}
\item $f'(2k+1)=f'(2k+2)=a$, and
\item $f'(i) = f(i)$ otherwise.
\end{itemize}

In each case, we can check that $f'(2k-1)\lp f'(2k+1)$ and $f'(2k+2) \lp f'(2k+3)$, so $f'$ satisfies all the conditions of \eqref{eq:g}, and is indeed a $P$-partition for $\pi'$.

The functions $f'$ constructed in the first case have $f'(2k+1)\neq f'(2k+2)$, so the two cases do not overlap, yielding the desired injection $f\mapsto f'$.
\end{proof}

\subsection{Consequences for shuffling probabilities}

We now connect the results for $P$-partitions to shuffling probabilities. To begin, let $m$ and $n$ be positive integers, and let $\pi$ be a permutation in $S_n$. Consider an $m$-shelf shuffler, and define
\begin{itemize}
\item $x_m(\pi)$ to be the probability of obtaining $\pi$ in lazy mode,
\item $x_m^*(\pi)$ to be the probability of obtaining $\pi$ in standard mode, and
\item $x_m^+(\pi)$ to be the probability of obtaining $\pi$ in strict mode.
\end{itemize}
Similarly, we define $y_m(\pi)$, $y_m^*(\pi)$, and $y_m^+(\pi)$ to be probabilities of obtaining $\pi$ from an $m$-up-down riffle shuffle, an $m$-down-up riffle shuffle, and a classic $m$-riffle shuffle, respectively. Our choice of notation is suggestive of the following exact formulas for these probabilities in terms of order polynomials. We believe this to be new in the lazy case.

\begin{prop}\label{prp:shufprobs}
For each permutation $\pi \in S_n$ and each positive integer $m$, we have the following expressions for shuffling probabilities:
\begin{align}
x_m(\pi) &= \frac{\op_{\pi}(m)}{(2m+1)^n},\label{eq:prob1}\\
x^*_m(\pi) &= \frac{ \op^*_{\pi}(m) }{(2m)^n},\label{eq:prob2}\\
x_m^+(\pi) &= \frac{ \op^+_{\pi}(m) }{m^n}.\label{eq:prob3}
\end{align}
Moreover, the probabilities for shelf-shuffling and riffle shuffling are related via
\[
 y_m(\pi) = x_m(\pi^{-1}), \quad y_m^*(\pi) = x_m^*(\pi^{-1}), \quad \mbox{ and } \quad y^+_m(\pi) = x^+_m(\pi^{-1}).
\]
\end{prop}

In essence, this result says that the probability of obtaining $\pi$ from shelf shuffling is the probability that a random $P$-partition has sorting permutation $\pi$, and the probability of obtaining $\pi$ from riffle shuffling is the probability that a random $P$-partition has sorting permutation $\pi^{-1}$.

\begin{proof}
By Proposition \ref{prp:equiv}, we know that choosing a random $P$-partition in $\A(P;m)$ is equivalent to a random outcome of an $m$-shelf shuffler in lazy mode. By Observation \ref{obs:antichainop}, there are
$\op_{[n]}(m) = (2m+1)^n$ such outcomes. By Definition \ref{def:permf} and the definition of the order polynomial, precisely $\op_{\pi}(m)$ of these correspond to the permutation $\pi$. This proves Equation \eqref{eq:prob1} for $x_m(\pi)$.

The fact that $x_m(\pi)=y_m(\pi^{-1})$ is an immediate consequence of Proposition \ref{prp:equiv2}.

The arguments proving \eqref{eq:prob2} and \eqref{eq:prob3} (for standard and strict shuffling modes) are similar.
\end{proof}

We next consider, in each mode (lazy, standard, strict), a generating function for the entire probability distribution as an element in the group algebra of the symmetric group. That is, define
\begin{align*}
 \phi_n(m) &=\sum_{\pi \in S_n} x_m(\pi)\pi,\\
 \phi^*_n(m) &=\sum_{\pi \in S_n} x^*_m(\pi)\pi, \mbox{ and}\\
\phi^+_n(m) &=\sum_{\pi \in S_n} x^+_m(\pi)\pi.
\end{align*}

As a corollary to Proposition \ref{prp:decomp}, we get the following identities for the distributions. (These identities do not require that $k,l \geq 0$, though our proof does).

\begin{cor}\label{cor:convolution}
For each $n\geq 1$ and $k, l\geq 0$, we have
\begin{align}
 \phi_n(k)\phi_n(l) &= \phi_n(2kl+k+l),\\
 \phi^*_n(k)\phi^*_n(l) &= \phi^*_n(2kl), \mbox{ and}\\
 \phi^+_n(k)\phi^+_n(l) &= \phi^+_n(kl).
\end{align}
\end{cor}

\begin{proof}
We handle the lazy mode case in detail. Other shuffling modes are similar.

We have
\begin{align*}
 \phi_n(k)\phi_n(l) &= \left(\sum_{\sigma \in S_n} x_k(\sigma)\sigma\right)\left(\sum_{\tau \in S_n} x_l(\tau)\tau\right),\\
 &=\left(\sum_{\sigma \in S_n} \frac{\op_{\sigma}(k)}{(2k+1)^n}\sigma\right)\left(\sum_{\tau \in S_n} \frac{\op_{\sigma}(l)}{(2l+1)^n}\tau\right),\\
 &=\frac{1}{(4kl+2k+2l+1)^n}\sum_{\pi\in S_n} \left( \sum_{\sigma\tau=\pi} \op_{\sigma}(k)\op_{\tau}(l)\right)\pi.
\end{align*}
But equation \eqref{eq:decomp1} gives
\[
 \sum_{\sigma\tau=\pi} \op_{\sigma}(k)\op_{\tau}(l)=\op_{\pi}(2kl+k+l),
\]
so we obtain
\begin{align*}
\phi_n(k)\phi_n(l) &=\frac{1}{(2(2kl+k+l)+1)^n}\sum_{\pi\in S_n} \op_{\pi}(2kl+k+l)\pi,\\
&=\sum_{\pi \in S_n} x_{2kl+k+l}(\pi)\pi,\\
&=\phi_n(2kl+k+l),
\end{align*}
as claimed.
\end{proof}

The immediate consequence of these identities has to do with repeated shuffles. For example, the distribution after two lazy $m$-shuffles is $\phi(m)^2$, and Corollary \ref{cor:convolution} tells us that $\phi(m)^2 = \phi(2m^2+2m)$, so two sequential lazy $m$-shelf shuffles gives the same distribution as one lazy $(2m^2+2m)$-shuffle. This means two sequential lazy $10$-shelf shuffles give the same distribution as one pass through a lazy $220$-shelf shuffler.

\section{Convergence results}\label{sec:converge}

In both \cite{BD} and \cite{DFH} we get estimates for how quickly shuffling converges to the uniform distribution on $S_n$. We follow those papers in considering the following measures for any probability distribution $\Prob$ on $S_n$. We let $U$ be the uniform distribution, so that $U(\pi)=1/n!$ for all $\pi$ in $S_n$, and define the \emph{total variation distance}
\[
 \|\Prob-U\|_{TV} =\frac{1}{2}\sum_{\pi \in S_n} |\Prob(\pi) - U(\pi)| = \frac{1}{2}\sum_{\pi \in S_n} \left|\Prob(\pi) - \frac{1}{n!} \right|,
\]
the \emph{separation distance}
\[
 \sep(\Prob) =\max_{\pi \in S_n} \left( 1- \frac{\Prob(\pi)}{U(\pi)} \right) =\max_{\pi \in S_n} \left( 1- n!\Prob(\pi)\right),
\]
and \emph{$l_{\infty}$ distance}
\[
 \|\Prob-U\|_{\infty} = \max_{\pi \in S_n} \left\vert 1- \frac{\Prob(\pi)}{U(\pi)} \right\vert = \max_{\pi \in S_n} \left\vert 1- n!\Prob(\pi)\right\vert.
\]
It is elementary that $\|\Prob-U\|_{TV} \leq \sep(\Prob) \leq \|\Prob-U\|_{\infty}$.

Let
\[
 x_m(k) = x_m(\pi) \mbox{ for some $\pi$ with $\lpk(\pi)=k$,}
\]
and similarly define $x_m^*(k)$ in terms of peaks and $x_m^+(k)$ in terms of descents.

Taking the expressions for order polynomials from Proposition \ref{prp:chains}, we find the following expressions for our probabilities:
\begin{align*}
 x_m(k) &= \frac{4^k}{(2m+1)^n}\sum_{a\geq 0} \binom{n+m-a}{n}\binom{n-2k}{a-k}, \\
 x_m^*(k) &= \frac{4^{k+1}}{2(2m)^n}\sum_{a\geq 0} \binom{n-1+m-a}{n}\binom{n-1-2k}{a-k},
\mbox{ and}\\
 x_m^+(k) &= \frac{1}{m^n}\binom{n-1+m-k}{n}.
\end{align*}

For example in the lazy case, our total variation distance can be expressed as
\begin{equation}\label{eq:TV}
 \|x_m-U\|_{TV} =\frac{1}{2}\sum_{k=0}^{\lfloor n/2 \rfloor} l(n,k)\left|x_m(k) - \frac{1}{n!}\right|,
\end{equation}
where $l(n,k)=|\{\pi \in S_n : \lpk(\pi) = k\}|$. From \cite{Ma}, the numbers $l(n,k)$ satisfy the recurrence
\[
l(n,k)=(2k+1)l(n-1,k) + (n+1-2k)l(n-1,k-1),
\]
with boundary conditions $l(n,0)=1$ and $l(n,k)=0$ if $k>n/2$. The formulas make it not too difficult to use Equation \eqref{eq:TV} to compute the total variation distance for realistic values of $m$ and $n$. For example, in Table \ref{tab:TV}, we see total variation distance for various values of $m$ and $n=52$, comparing $x_m$ (lazy), $x_m^*$ (standard), and $x_m^+$ (strict). We note also that if $m=10$ and we pass through the lazy shuffler twice it is the same as $m=220$. In this case $\|x_m-U\|_{TV}=.0083$. The table also shows that lazy shuffling usually does better than standard, but for $m=30$ and $m=35$, standard does better; we do not have an explanation for this.

\begin{table}
{\small
\begin{tabular}{|c| cccccccccccc|}
\hline
 $m$ & 10 & 15 & 20 & 25 & 30 & 35 & 50 & 100 & 150 & 200 & 250 & 300\\
\hline
\hline
 Lazy  & 1 & .9372 & .7184 & .5164 & .3936 & .3003 & .1509 & .0392 & .0177 & .0100 & .0064 & .0045 \\
\hline
 Standard & 1 & .9427 & .7201 & .5440 & .3910 & .2993 & .1586 & .0409 & .0183 & .0103 & .0066 & .0046 \\
\hline
 Strict & 1 & 1 & .9981 & .9825 & .9468 & .8932 & .7336 & .4199 & .2857 & .2131 & .1709 & .1438\\
 \hline
\end{tabular}
}
\bigskip
\caption{Total variation distance for shelf shufflers with $m$ shelves and $n=52$ cards, in each of the three operating modes.}\label{tab:TV}
\end{table}

The monotonicity lemma for order polynomials, Lemma \ref{lem:monotone}, implies that $x_m(k) \geq x_m(k+1)$, $x_m^*(k)\geq x_m^*(k+1)$, and $x_m^+(k)\geq x_m^+(k+1)$. Thus, we see that both the $l_{\infty}$ and separation distances are achieved at the extremes.

\begin{obs}\label{obs:extremes}
For any distribution $\Prob \in \{x_m, x_m^*, x_m^+\}$, we have
\[
 \sep(\Prob) = \max\{ 1-n!\Prob(0), 1-n!\Prob(k_{\max})\},
\]
and
\[
 \| \Prob-U\|_{\infty} = \max\{ |1-n!\Prob(0)|, |1-n!\Prob(k_{\max})|\},
\]
where $k_{\max} = \lfloor n/2\rfloor, \lfloor (n-1)/2\rfloor,$ and $n-1$, for $x_m$, $x_m^*$, and $x_m^+$, respectively.
\end{obs}

The $l_{\infty}$ and separation distances are easy to study for strict shelf shufflers using
the explicit formula for $x_m^+(\pi)$. More subtle calculations are required for standard shelf
shufflers \cite{DFH}. For lazy shelf shufflers the asymptotics are the same as for standard shelf
shufflers. More precisely, we have the following result, which shows that for $n$ cards, order $n^{3/2}$ shelves are necessary and sufficient for randomness.

\begin{thm} Consider the lazy shelf shuffling measure $x_m$ with $n$ cards and $m$ shelves
(and the additional shelf $0$). Suppose $m=cn^{3/2}$. Then as $n \rightarrow \infty$ with $0 < c < \infty$ fixed,
\[ ||x_m-U||_{\infty} \sim e^{1/(12c^2)} - 1 \]
\[ \sep(x_m) \sim 1 - e^{-1/(24c^2)} \]
\end{thm}

\begin{proof} The proof is a very minor modification of arguments in \cite{DFH}.
Assume $n$ to be even (for simplicity). Then by our formula for $x_m(k)$ and Observation \ref{obs:extremes}, we know
the extreme values are
\[ x_m(0) = \frac{1}{(2m+1)^n} \sum_{a \geq 0} {n+m-a \choose n}  {n \choose a}, \]
while
\[ x_m(n/2) = \frac{1}{(m+1/2)^n} {m + n/2 \choose n} .\]

Arguing as in \cite{DFH} (to which the reader is referred for all the analytic details), we have that when $m=cn^{3/2}$ and $n \rightarrow \infty$ with $0 < c < \infty$ fixed,
\[ \frac{n!}{(2m+1)^n} \sum_{a \geq 0} {n+m-a \choose n}  {n \choose a} -1 \sim
e^{1/(12c^2)} - 1 \]
\[ 1 - \frac{n!}{(m+1/2)^n} {m + n/2 \choose n} \sim 1 - e^{-1/(24c^2)}, \] and the result follows. \end{proof}

\section{Cycle structure for lazy shelf shufflers}\label{sec:cycle}

The cycle structure of strict shelf shufflers is the same as the cycle structure of ordinary
riffle shuffles, carefully studied in \cite{DMP}. The cycle structure of standard shelf shufflers
is studied in \cite{DFH}. In this section we find a generating function for cycle
structure of lazy shelf shufflers. Then we use it to derive the joint distribution of permutations
by cycles and left peaks.

Let $N_i(\pi)$ denote the number of $i$-cycles of a permutation $\pi$, and define
\[ f_{i,m} = \frac{1}{2i} \sum_{d|i \atop d \ \textrm{odd}} \mu(d) [(2m+1)^{i/d}-1], \] where
$\mu$ is the M\"obius function of elementary number theory. Let $x_m(\pi)$ and
$y_m(\pi)$ be as in previous sections.

\begin{thm} \label{cycstruc}
\begin{eqnarray*}
& & 1 + \sum_{n \geq 1} u^n \sum_{\pi \in S_n} x_m(\pi) \prod_{i \geq 1} z_i^{N_i(\pi)} \\
& = & \frac{1}{1-z_1 u/(2m+1)} \prod_{i \geq 1} \left( \frac{1+z_iu^i/(2m+1)^i}
{1-z_i u^i/(2m+1)^i} \right)^{f_{i,m}}.
\end{eqnarray*}
\end{thm}

\begin{proof} It follows from Theorem 7 of \cite{F1} that
\begin{eqnarray*}
& & 1 + \sum_{n \geq 1} u^n \sum_{\pi \in S_n} y_m(\pi) \prod_{i \geq 1} z_i^{N_i(\pi)} \\
& = & \frac{1}{1-z_1 u/(2m+1)} \prod_{i \geq 1} \left( \frac{1+z_iu^i/(2m+1)^i}
{1-z_i u^i/(2m+1)^i} \right)^{f_{i,m}}.
\end{eqnarray*} Note that a permutation and its inverse have the same cycle structure, and recall from
Proposition \ref{prp:shufprobs} that $x_m(\pi)=y_m(\pi^{-1})$. The result follows.
\end{proof}

The generating function in Theorem \ref{cycstruc} allows one to study cycle structure for lazy shelf-shufflers, in perfect analogy with the results of \cite{DMP} for ordinary riffle shuffles. The following proposition illustrates this.

\begin{prop} \label{fix} The average number of fixed points after a lazy shelf shuffle with $n$ cards and $m$ shelves is equal to
\begin{eqnarray*}
 1 + 2 \sum_{k=1}^{(n-1)/2} \frac{1}{(2m+1)^{2k}} &  \mbox{if $n$ is odd,} \\
 1 + 2 \sum_{k=1}^{n/2-1} \frac{1}{(2m+1)^{2k}} + \frac{1}{(2m+1)^n} &  \mbox{if $n$ is even.}
\end{eqnarray*}
\end{prop}

\begin{proof} Set $z_1=z$ and all other $z_i=1$ in Theorem \ref{cycstruc}. The right hand side becomes
\[ \frac{1}{1-z u/(2m+1)} \left( \frac{1+zu/(2m+1)}{1-zu/(2m+1)} \right)^{f_{1,m}} \prod_{i \geq 2} \left( \frac{1+u^i/(2m+1)^i}
{1-u^i/(2m+1)^i} \right)^{f_{i,m}}. \]

Setting all $z_i=1$ in Theorem \ref{cycstruc} gives that
\[ \frac{1}{1-u} = \frac{1}{1-u/(2m+1)} \prod_{i \geq 1} \left( \frac{1+u^i/(2m+1)^i}
{1-u^i/(2m+1)^i} \right)^{f_{i,m}}.\]

Combining the previous two paragraphs, one concludes that
\[ 1 + \sum_{n \geq 1} u^n \sum_{\pi \in S_n} x_m(\pi) z^{N_1(\pi)} \] is equal to
\[ \frac{1}{1-u} \left( \frac{1-u/(2m+1)}{1-zu/(2m+1)} \right)^{m+1}
\left( \frac{1+zu/(2m+1)}{1+u/(2m+1)} \right)^m .\]

Differentiating with respect to $z$ and setting $z=1$ shows that the expected number of fixed points is the coefficient of $u^n$ in
\[ \frac{1}{1-u} \left[ \frac{(m+1) u/(2m+1)}{1-u/(2m+1)} +
\frac{mu/(2m+1)}{1+u/(2m+1)} \right], \] and the proposition easily follows from
this. \end{proof}

\begin{rem} In addition to the intrinsic interest of fixed points, Proposition \ref{fix} shows that the expected number of fixed points is close to $1$ when $m$ tends to infinity arbitrarily slowly with $n$. In fact Theorem \ref{cycstruc} can be used to show that the entire distribution of fixed points tends to a Poisson(1) limit when $m$ tends to infinity arbitrarily slowly with $n$. Thus the number of shelves for the distribution of fixed points to be close to that of a uniform permutation is far fewer than the order $n^{3/2}$ shelves needed to randomize the entire deck.

Moreover, these ideas imply a lower bound saying that $m$ must go to infinity with $n$ in order for $x_m$ to be close to the uniform distribution. Indeed, for any probability distributions $P$ and $Q$ on a finite set $X$,
\[ \|P-Q\|_{TV} = \max_{A \subseteq X} |P(A)-Q(A)|.\] So if $A$ is any subset of the state space $X$, \[ \|P-Q\|_{TV} \geq |P(A)-Q(A)|.\] So if the distribution of fixed points is far from random, then $x_m$ is not close to the uniform distribution.

To get shaper total variation lower bounds, one should study the distribution of the number of left peaks under the distribution $x_m$. In particular, since $x_m$ is supported on permutations with at most $m$ left peaks, it's clear that if $m$ is fixed then for large $n$, the distribution $x_m$ is far from uniform.
\end{rem}

\begin{rem} In contrast to the distribution of the number of fixed points, one might want to study features of large cycles, such as the length of the longest cycle, under the distribution $x_m$. Using Theorem \ref{cycstruc}, one can prove that as $n \rightarrow \infty$, the distribution of the length of the longest cycle under $x_m$ is close to the distribution of the length of the longest cycle under the uniform distribution, even if $m=1$.
\end{rem}

Our final result gives a generating function for the joint distribution of permutations
by number of left peaks and cycle type. This result appeared in a recent paper of Gessel
and Zhuang (\cite{GeZh}, Theorem 7.2), though our proof is completely different, and we discovered it
 independently. It is an analog of a result in \cite{DFH} which gave a
generating function for the joint distribution of permutations by number of peaks and cycle type,
and of a result of \cite{Fu2} giving a generating function for the
joint distribution of permutations by number of descents and cycle type.

\begin{cor}
\begin{eqnarray*}
& & \frac{t}{1-t} + \sum_{n \geq 1} u^n \sum_{\pi \in S_n} \frac{(1+t)^n}{(1-t)^{n+1}}
\left( \frac{4t}{(1+t)^2} \right)^{\lpk(\pi)} \prod_{i \geq 1} z_i^{N_i(\pi)} \\
& = & \sum_{m \geq 1} t^m \frac{1}{1-z_1 u} \prod_{i \geq 1} \left( \frac{1+z_iu^i}{1-z_iu^i} \right)^{f_{i,m}}.
\end{eqnarray*}
\end{cor}

\begin{proof} Take the coefficient of $t^m$ in both sides of the statement of the corollary,
and replace $u$ by $u/(2m+1)$.
 By Propositions \ref{prp:chains} and \ref{prp:shufprobs}, \[ x_m(\pi) = \frac{1}{(2m+1)^n} [t^m] \frac{(1+t)^n}{(1-t)^{n+1}} \left( \frac{4t}{(1+t)^2} \right)^{\lpk(\pi)},\] where
 $[t^m] f(t)$ denotes the coefficient of $t^m$ in a power series $f(t)$. The result follows
 from Theorem \ref{cycstruc}.
\end{proof}

\end{document}